\documentclass[10pt,reqno]{amsproc}

\usepackage{fullpage}
\usepackage[dvipsnames]{xcolor}
\usepackage[utf8]{inputenc}

\usepackage{
  amsmath, amsthm, amssymb, mathtools, dsfont, units,          
  graphicx, wrapfig, subfig, float,                            
  listings, pythonhighlight,                                  
  fancyhdr,
  hyperref,
  enumerate,
  enumitem,
  framed
}   

\usepackage{bm}

\usepackage{newpxtext, newpxmath, inconsolata}
\usepackage{amsfonts}
\usepackage{pgfplots}
\pgfplotsset{compat=1.12}
\usepackage{tkz-fct}
\usepackage{svg}
\usepackage{tikz}
\usepackage{tikz-cd}
\usepackage{rank-2-roots}
\usepackage{multicol}
\usepackage{multirow}

\hypersetup{colorlinks=true, linkcolor=red, citecolor=blue, filecolor=blue, urlcolor=blue}

\usepackage[bottom]{footmisc}

\allowdisplaybreaks

\usepackage[font={it,footnotesize}]{caption}

\usepackage{titlesec}
\titleformat{\section}{\large\bfseries}{\thesection\;\;\;}{0em}{}
\titleformat{\subsection}{\normalsize\bfseries\selectfont}{\thesubsection\;\;\;}{0em}{}
\titleformat{\subsubsection}{\normalsize\bfseries\selectfont}{\thesubsubsection\;\;\;}{0em}{}

\setlist[itemize]{wide=0pt, leftmargin=16pt, labelwidth=10pt, align=left}

\graphicspath{{Images/}{../Images/}}

\usepackage[mathscr]{euscript}

\newcommand{\what}[1]{\widehat{#1}}

\newcommand{\SL}{\mathrm{SL}}

\newcommand{\QM}{\mathcal{QM}}

\newcommand{\wtilde}{\widetilde}
\newcommand{\dd}{\mathrm{d}}

\newcommand{\rA}{\mathrm{A}}

\newcommand{\bC}{\mathbb{C}}

\newcommand{\bH}{\mathbb{H}}

\newcommand{\bR}{\mathbb{R}}

\newcommand{\bZ}{\mathbb{Z}}

\newcommand{\bx}{\mathbf{x}}
\newcommand{\by}{\mathbf{y}}

\newcommand{\cL}{\mathcal{L}}
\newcommand{\cM}{\mathcal{M}}
\newcommand{\cN}{\mathcal{N}}

\newcommand{\cS}{\mathcal{S}}

\theoremstyle{definition}
\newtheorem{theorem}{Theorem}[section]
\newtheorem{proposition}[theorem]{Proposition}
\newtheorem{lemma}[theorem]{Lemma}
\newtheorem{corollary}[theorem]{Corollary}
\newtheorem{definition}[theorem]{Definition}

\theoremstyle{remark}
\newtheorem{remark}[theorem]{Remark}

\newtheorem*{theorem*}{Theorem}
\newtheorem*{proposition*}{Proposition}
\newtheorem*{lemma*}{Lemma}
\newtheorem*{corollary*}{Corollary}
\newtheorem*{definition*}{Definition}
\newtheorem*{example*}{Example}
\newtheorem*{remark*}{Remark}

\newtheorem*{conjecture*}{Conjecture}

\definecolor{keywordcolor}{rgb}{0.7, 0.1, 0.1}   
\definecolor{tacticcolor}{rgb}{0.0, 0.1, 0.6}    
\definecolor{commentcolor}{rgb}{0.4, 0.4, 0.4}   
\definecolor{symbolcolor}{rgb}{0.0, 0.1, 0.6}    
\definecolor{sortcolor}{rgb}{0.1, 0.5, 0.1}      
\definecolor{attributecolor}{rgb}{0.7, 0.1, 0.1} 
\definecolor{backcolour}{rgb}{0.95,0.95,0.92}
\colorlet{shadecolor}{backcolour}

\BeforeBeginEnvironment{lstlisting}{%
  \begingroup
  \setlength{\fboxsep}{0pt}%
  \setlength{\OuterFrameSep}{0pt}%
  \begin{snugshade}%
}
\AfterEndEnvironment{lstlisting}{%
  \end{snugshade}%
  \endgroup
}

\newenvironment{red}{\relax\color{red}}{\hspace*{.5ex}\relax}
\newenvironment{blue}{\relax\color{blue}}{\hspace*{.5ex}\relax}
\newcommand{\ber}{\begin{red}}
\newcommand{\er}{\end{red}}
\newcommand{\beb}{\begin{blue}}
\newcommand{\eb}{\end{blue}}

\title{Positive quasimodular forms and the sign uncertainty principle}
\author{Seewoo Lee}
\date{}

\begin{document}

\begin{abstract}
    For every positive integer $d$ divisible by $4$, we prove the following new upper bound for the Bourgain--Clozel--Kahane sign uncertainty constant:
    \[
        \rA_+(d) \le \sqrt{2 \left\lfloor \frac{d}{16} \right\rfloor + 2}.
    \]
    It recovers the optimal bound $\rA_+(12) \le \sqrt{2}$ in dimension $12$ and improves the previously best known bound $\sqrt{(d+2)/(2\pi)}$ for all $d \ge 52$ divisible by $4$.
    The proof uses Fourier eigenfunctions and associated quasimodular forms constructed by Feigenbaum, Grabner, and Hardin.
\end{abstract}

\maketitle

\section{Introduction}

An uncertainty principle describes a trade-off between a function and its Fourier transform.
One such formulation is the Bourgain--Clozel--Kahane uncertainty principle for the last sign change \cite{bourgain2010principe}.
Let $f: \bR^d \to \bR$ be a ``nice function'' that is eventually nonnegative, meaning that $f(\bx) \ge 0$ for all sufficiently large $\|\bx\|$.
Let $\what{f}: \bR^d \to \bC$ be the Fourier transform of $f$, given by
\[
\what{f}(\by) = \int_{\bR^d} f(\bx) e^{-2\pi i \langle \bx, \by \rangle} \, \dd \bx.
\]
We consider the class $\mathcal{A}_+(d)$ of functions satisfying the following conditions:
\begin{enumerate}
    \item $f \in L^1(\bR^d)$, $\what{f} \in L^1(\bR^d)$, and $\what{f}$ is real-valued,
    \item $f$ is eventually nonnegative while $\what{f}(\mathbf{0}) \le 0$, and
    \item $\what{f}$ is eventually nonnegative while $f(\mathbf{0}) \le 0$.
\end{enumerate}
Let $r(f)$ be the last-sign-change radius of $f$:
\[
r(f) := \inf \{ r \ge 0 : f(\bx) \ge 0 \text{ whenever } \|\bx\| \ge r \},
\]
and define $r(\what{f})$ similarly.
The uncertainty principle of Bourgain--Clozel--Kahane says \cite[Th\'eor\`eme 3.1]{bourgain2010principe}
\[
\rA_{+}(d) := \inf_{f \in \mathcal{A}_+(d) \setminus \{0\}} \sqrt{r(f) r(\what{f})} > 0.
\]

The natural next question is to determine the optimal constant $\rA_+(d)$, or at least effective lower and upper bounds for it.
Its exact value is known only in dimension $d = 12$, where Cohn and Gon{\c{c}}alves proved that $\rA_+(12) = \sqrt{2}$ \cite{cohn2019optimal}.
They adapted the magic-function constructions underlying the optimal sphere packings in dimensions $8$ and $24$ \cite{viazovska2017sphere,cohn2017sphere}.
In particular, the proof of the upper bound $\rA_+(12) \le \sqrt{2}$ is based on the construction of the optimal function $f$ as a certain integral transform of a modular form.
The best previously known explicit \emph{uniform} upper bound, valid for all $d \ge 2$, is due to Bourgain--Clozel--Kahane \cite{bourgain2010principe}:
\begin{equation}
    \label{eqn:bckupper}
    \rA_+(d) \le \sqrt{\frac{d + 2}{2\pi}}.
\end{equation}
More recently, an internal OpenAI model proved the asymptotic bound \cite[Chapter 1, Theorem 1.2]{openai10}
\begin{equation}
    \label{eqn:Apos_limit}
    \rA_{+}(d) \le \left(\frac{1}{\pi} + o(1)\right) \sqrt{d},
\end{equation}
as $d \to \infty$, by proving that $\lim_{d \to \infty} \rA_{+}(d)/\sqrt{d} = 1/\pi$.

It is natural to ask whether the construction of Cohn--Gon{\c{c}}alves for $d = 12$ can be generalized to other dimensions and thereby yield upper bounds for $\rA_+(d)$ that improve on \eqref{eqn:bckupper}.
Feigenbaum, Grabner, and Hardin \cite{feigenbaum2021eigenfunctions} constructed a family of Fourier eigenfunctions in dimensions $d \equiv 0 \pmod{4}$ that subsumes the previous constructions in \cite{viazovska2017sphere,cohn2017sphere,cohn2019optimal}.
They treated several specific dimensions and proved new upper bounds for $\rA_+(d)$, but their argument does not extend directly to all dimensions because it relies on dimension-specific numerical computations.

\begin{figure}[t]
\centering
\begin{tikzpicture}

\begin{axis}[
    name=mainP,
    width=15cm,
    height=7.5cm,
    xlabel={$d$},
    ylabel={$\rA_+(d)$},
    xmin=1, xmax=10000,
    ymin=0, ymax=42,
    grid=both,
    tick label style={font=\small},
    label style={font=\small},
    scaled x ticks=false,
    every axis plot/.append style={line cap=round},
]

\addplot[thick, cyan!60, domain=1:10000, samples=80]
  {sqrt((x+2)/(2*pi))};
\addplot[blue!90!black, only marks, mark=triangle*, mark size=0.9pt, samples at={4,24,...,9984}]
  {sqrt(2*floor(x/16)+2)};

\addplot[thick, dotted, gray, domain=1:10000, samples=200]
  {sqrt(x) / pi};

\addplot[draw=black, fill=yellow!25, fill opacity=0.35, thick]
  coordinates {(1,0) (100,0) (100,6) (1,6) (1,0)};
\coordinate (zoomTRP) at (axis cs:100,6);
\coordinate (zoomBRP) at (axis cs:100,0);

\end{axis}

\begin{axis}[
    name=insetP,
    at={(mainP.south east)},
    anchor=south east,
    xshift=-0.3cm,
    yshift=0.5cm,
    width=5.8cm,
    height=3.8cm,
    xmin=1, xmax=100,
    ymin=0, ymax=6,
    grid=both,
    tick label style={font=\scriptsize},
    label style={font=\scriptsize},
    axis background/.style={fill=white},
    every axis plot/.append style={line cap=round},
]

\addplot[thick, cyan!60, domain=1:100, samples=100]
  {sqrt((x+2)/(2*pi))};
\addplot[teal!90, only marks, mark=*, mark size=0.5pt, line width=0.55pt]
  table[x=d,y=Aplus] {src/cohn_goncalves_aplus.dat};
\addplot[blue!90!black, only marks, mark=triangle*, mark size=0.8pt, samples at={4,8,...,100}]
  {sqrt(2*floor(x/16)+2)};

\end{axis}

\draw[black, thin] (zoomTRP) -- (insetP.north west);
\draw[black, thin] (zoomBRP) -- (insetP.south west);

\path (mainP.south west) -- coordinate[midway] (mainPcenter) (mainP.south east);
\matrix[matrix of nodes, anchor=north, row sep=2mm, column sep=6mm,
        nodes={font=\scriptsize, anchor=west}] at ([yshift=-1.2cm]mainPcenter) {
  \tikz{\draw[thick, cyan!60] (0,0)--(0.45,0);} & Bourgain--Clozel--Kahane \eqref{eqn:bckupper} &
  \tikz{\fill[blue!90!black] (0.225,0) +(-1.5pt,0) -- +(1.5pt,0) -- +(0,2.5pt) -- cycle;} & New upper bound \eqref{eqn:main} \\
  \tikz{\draw[thick, dotted, gray] (0,0)--(0.45,0);} & Asymptotic scale from \eqref{eqn:Apos_limit} &
  \tikz{\fill[teal!90] (0.225,0) circle (0.8pt);} & Cohn--Gon{\c{c}}alves numerical bounds \cite[Table 2]{cohn2019optimal} \\
};

\end{tikzpicture}
\caption{Upper bounds on $\rA_+(d)$, with a zoomed view for $d \le 100$ in the inset. The teal dots are the numerical polynomial--Gaussian bounds of Cohn--Gon{\c{c}}alves for $1\le d\le32$ \cite[Table 2]{cohn2019optimal}. The new bound proved in Theorem~\ref{thm:main} (blue triangles) is sampled at every multiple of $4$ in the inset; the main plot shows representative samples.
The dotted curve shows the asymptotic scale in \eqref{eqn:Apos_limit}, rather than a finite-dimensional bound.}
\label{fig:bounds_Aplus}
\end{figure}
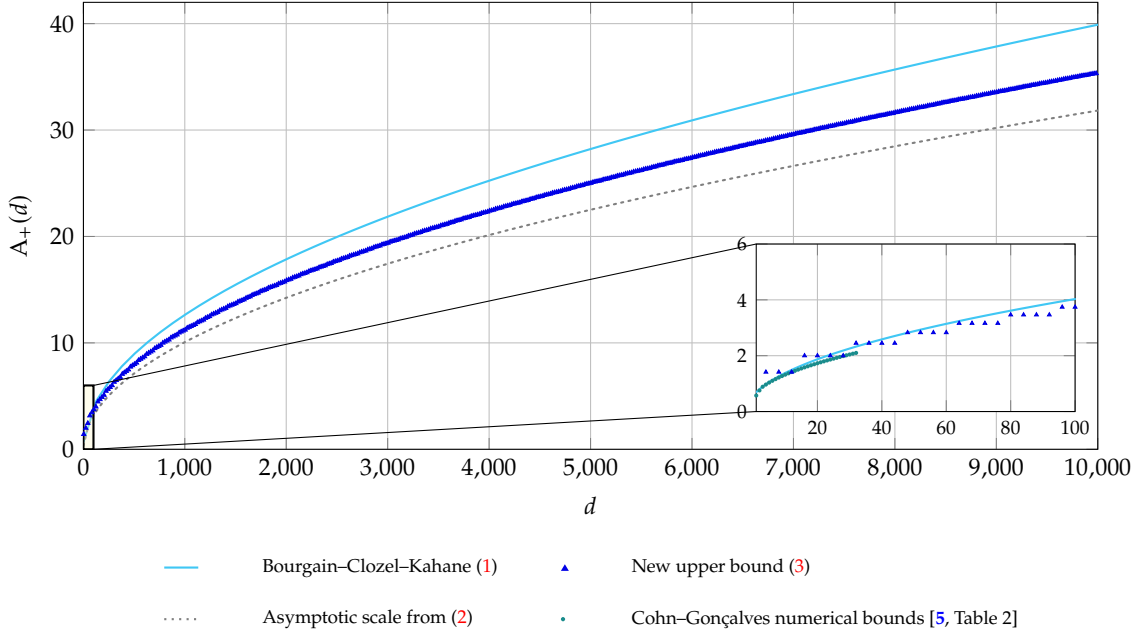

In this paper, we prove that the quasimodular forms associated with these eigenfunctions are positive in every dimension $d \equiv 0 \pmod{4}$. This yields the following new upper bound for $\rA_+(d)$.

\begin{theorem}
\label{thm:main}
For each positive integer $d$ divisible by $4$, we have
\begin{equation}
    \label{eqn:main}
    \rA_+(d) \le \sqrt{2 \left\lfloor \frac{d}{16} \right\rfloor + 2}.
\end{equation}
Moreover, the inequality is strict unless $d = 12$.
\end{theorem}

For $d$ divisible by $8$, the main idea is to relate the quasimodular-form family corresponding to the Fourier eigenvalue $(-1)^{d/4}$ to the \emph{extremal quasimodular forms} of Kaneko and Koike \cite{kaneko2006extremal}, which are conjectured to have positive Fourier coefficients.
We prove new recurrence relations and identities that reduce the required positivity to that of the extremal quasimodular forms of depth $2$.
The latter positivity follows from new recurrence relations and Nakaya's hypergeometric identity \cite{nakaya2024determination}; Proposition \ref{prop:ext2poseig} gives the identities relating the Feigenbaum--Grabner--Hardin forms to these extremal forms.
When $d \equiv 4 \pmod{8}$, we define an analogous family of ``modular forms of level $2$'' satisfying similar identities with the family corresponding to the Fourier eigenvalue $(-1)^{d/4+1}$.
The resulting argument recovers the optimal bound $\rA_+(12) \le \sqrt{2}$ when $d = 12$.

\subsection*{Acknowledgements}
This work is part of the author's PhD thesis.
The author thanks Paata Ivanisvili and Sug Woo Shin for helpful discussions and comments.

\subsection*{Disclosure of AI usage}
The proofs of Lemma \ref{lem:intertwine}, Proposition \ref{prop:Ftilde_pos}, and Proposition \ref{prop:Gtilde_pos_boundary} were developed with assistance from \texttt{ChatGPT-5.6 Sol}.
To obtain Lemma \ref{lem:intertwine}, the author asked the model to determine the general conditions on the parameters under which the intertwining relation \eqref{eqn:intertwine} holds; this relation is used in Propositions \ref{prop:Yproperties} and \ref{prop:Gtilde_ode}.
For Proposition \ref{prop:Ftilde_pos}, the author supplied the definitions and recurrence relations from Theorem \ref{thm:fghposeig} and asked the model to prove nonnegativity of coefficient $\tilde{a}_{n,+}^{(w-2)}$ of $\wtilde{F}_{w-2}$ when $1 \le n \le \frac{w}{4} - 2$.
For Proposition \ref{prop:Gtilde_pos_boundary}, the author supplied the proof that the $n$-th Fourier coefficient of $\wtilde{G}_{w}$ is nonnegative for every $0 \le n \le \frac{w}{4}$ (Proposition \ref{prop:Gtilde_pos}) and asked the model to handle the boundary case $n = \frac{w}{4} + 1$.

The paper was proofread and revised with assistance from AI tools, including \texttt{ChatGPT-5.6 Sol} and \texttt{Claude Opus 5 / Fable 5}.
In particular, \texttt{Fable 5} suggested a simplification of the base-case argument in the proof of Theorem \ref{thm:Ywpos}.
Figure \ref{fig:bounds_Aplus} was generated with assistance from \texttt{ChatGPT}.

AI tools, primarily \texttt{Claude Opus 5 / Fable 5}, also assisted in writing the Lean and Sage code.
Several results in this paper were formalized in Lean 4 \cite{moura2021lean}: the results on Kaneko--Zagier operators (Lemmas \ref{lem:KZ2_coeff}, \ref{lem:KZ3_coeff}, and \ref{lem:intertwine}) and the coefficient-positivity results for $\wtilde{F}_{w-2}$ and $\wtilde{G}_w$ (Propositions \ref{prop:Ftilde_pos}, \ref{prop:Gtilde_pos}, and \ref{prop:Gtilde_pos_boundary}). See Appendix \ref{subsec:lean} for details.
The AI tools also assisted in writing Sage \cite{sagemath} code used to perform computational checks of the recurrence relations and identities for the quasimodular forms $F_w$, $\wtilde{F}_{w-2}$, $G_w$, $Y_w$, and $\wtilde{G}_w$.
See Appendix \ref{subsec:sage} for details.
No other results in this paper were obtained with AI assistance.

\section{Preliminaries}

\subsection{Quasimodular forms}

For any function $f: \mathbb{H} \to \mathbb{C}$ and integer $k$, we define the weight-$k$ slash action of $\gamma \in \mathrm{SL}_2(\mathbb{Z}) = \Gamma(1)$ by
\[
    (f|_{k} \gamma)(z) := (cz + d)^{-k}f \left(\frac{az + b}{cz + d}\right), \quad \gamma = \begin{bmatrix} a & b \\ c & d \end{bmatrix} \in \mathrm{SL}_{2}(\mathbb{Z}).
\]
We denote by $T = \left[\begin{smallmatrix} 1 & 1 \\ 0 & 1\end{smallmatrix}\right]$ and $S = \left[\begin{smallmatrix} 0 & -1 \\ 1 & 0\end{smallmatrix}\right]$ the standard generators of $\mathrm{SL}_{2}(\mathbb{Z})$.
Let $q = e^{2\pi i z}$ for $z \in \mathbb{H}$, the complex upper half plane.
For $n \ge 1$, let $\sigma_k(n):= \sum_{d \mid n} d^k$.
Define the Eisenstein series of weights $2$, $4$, and $6$ by
\begin{align}
    E_2 &= 1 - 24 \sum_{n\geq 1}\sigma_1(n)q^n, \label{eqn:e2fourier} \\
    E_4 &= 1 + 240 \sum_{n \geq 1} \sigma_3(n)q^n, \label{eqn:e4fourier} \\
    E_6 &= 1 - 504 \sum_{n \geq 1} \sigma_5(n)q^n. \label{eqn:e6fourier}
\end{align}
$E_4$ and $E_6$ are genuine modular forms, whereas $E_2$ is a \emph{quasi}modular form of weight $2$ and level $1$.
These series obey the following transformation laws:
\begin{align*}
    (E_{2}|_{2}S)(z) = z^{-2}E_2\left(-\frac{1}{z} \right) &= E_2(z) - \frac{6 i}{\pi z},\\
    (E_{4}|_{4}S)(z) = z^{-4}E_4\left(-\frac{1}{z} \right) &= E_4(z), \\
    (E_{6}|_{6}S)(z) = z^{-6}E_6\left(-\frac{1}{z} \right) &= E_6(z).
\end{align*}
The (graded) ring of quasimodular forms is isomorphic to a polynomial ring in three variables with generators $E_2$, $E_4$, and $E_6$ \cite{bruinier2008elliptic}, and we define the \emph{depth} of a quasimodular form as the highest degree of $E_2$ in its expression as a polynomial in $E_2$, $E_4$, and $E_6$.
This ring is closed under differentiation:
\begin{equation}
    \label{eqn:D_def}
    DF = F' := \frac{1}{2\pi i} \frac{\dd F}{\dd z}, \qquad \sum_{n} a_n q^n \longmapsto \sum_{n} na_n q^n,
\end{equation}
and differentiation increases the weight by $2$ and the depth by at most $1$.
For the Eisenstein series, we have Ramanujan's identities \cite{bruinier2008elliptic}
\begin{equation}
    \label{eqn:ramanujan}
    E_2' = \frac{E_2^2 - E_4}{12}, \quad
    E_4' = \frac{E_2E_4 - E_6}{3}, \quad
    E_6' = \frac{E_2 E_6 - E_4^2}{2}.
\end{equation}
We write $\QM_{w}^{s} = \QM_{w}^{s}(\Gamma(1))$ for the space of quasimodular forms of weight $w$ and depth at most $s$, and $\cM_{w} := \QM_{w}^{0}$ for the space of genuine modular forms of weight $w$.
We denote by $\Delta := (E_4^3 - E_6^2) / 1728$ the discriminant form, which is the unique normalized cusp form of weight $12$ on $\Gamma(1)$.
It can be expressed as an infinite product
\[
\Delta(z) = q \prod_{n \ge 1} (1 - q^n)^{24} = \eta(z)^{24},
\]
where $\eta(z) := q^{1/24} \prod_{n \ge 1} (1 - q^n)$ is the Dedekind eta function.

\subsection{Jacobi's theta functions and the modular $\lambda$-function}

Jacobi's theta functions are defined as
\begin{align*}
    \Theta_2(z) &= \sum_{n \in \mathbb{Z}} q^{\frac{1}{2}(n + \frac{1}{2})^2}, \\
    \Theta_3(z) &= \sum_{n \in \mathbb{Z}} q^{\frac{n^2}{2}}, \\
    \Theta_4(z) &= \sum_{n \in \mathbb{Z}} (-1)^{n}q^{ \frac{n^2}{2}}.
\end{align*}
These are weight $1/2$ modular forms of level $\Gamma(2)$.
Although the definition of half-integral weight modular forms is subtle (see, for example, \cite{shimura1973modular}), we only use the fourth powers of these forms, which are modular forms of weight $2$ and level $\Gamma(2)$.
We will denote them as $H_2$, $H_3$, and $H_4$, which admit Fourier expansions
\begin{align*}
    H_2(z) &= \Theta_2^4(z) = 2 \sum_{n \geq 0} r_4(2n + 1) q^{n + \frac{1}{2}}, \\
    H_3(z) &= \Theta_3^4(z) = 1 + \sum_{n \geq 1}  r_4(n)q^{\frac{n}{2}}, \\
    H_4(z) &= \Theta_4^4(z) = 1 + \sum_{n \geq 1} (-1)^{n} r_4(n)q^{\frac{n}{2}},
\end{align*}
where $r_4(k) := \# \{ \mathbf{x} \in \mathbb{Z}^{4}: \|\mathbf{x}\|^2 = k\}$.
Their transformations under $\mathrm{SL}_2(\mathbb{Z})$ are
\begin{align*}
    (H_{2}|_{2}T)(z) = -H_{2}(z), \quad (H_{3}|_{2}T)(z) &= H_{4}(z), \quad (H_{4}|_{2}T)(z) = H_{3}(z), \\
    (H_{2}|_{2}S)(z) = -H_{4}(z), \quad (H_{3}|_{2}S)(z) &= -H_{3}(z), \quad (H_{4}|_{2}S)(z) = -H_{2}(z).
\end{align*}
Also, we have the Jacobi identity $H_{3} = H_{2} + H_{4}$.
These functions are related to the Eisenstein series and the discriminant form as
\begin{align}
    E_4 &= \frac{1}{2}(H_{2}^{2} + H_{3}^{2} + H_{4}^{2}) = H_{2}^{2} + H_{2}H_{4} + H_{4}^{2}, \label{eqn:e4theta} \\
    E_6 &= \frac{1}{2} (H_{2} + H_{3})(H_{3} + H_{4}) (H_{4} - H_{2}), \label{eqn:e6theta} \\
    \Delta &= \frac{1}{256} (H_{2}H_{3}H_{4})^2. \label{eqn:disctheta}
\end{align}

The modular $\lambda$-function
\begin{equation}
    \label{eqn:lambda_def}
    \lambda(z) = \frac{H_{2}(z)}{H_{3}(z)} = \frac{H_{2}(z)}{H_{2}(z) + H_{4}(z)},
\end{equation}
is a Hauptmodul for $\Gamma(2)$ and satisfies the transformation laws
\[
\lambda(z+1) = \frac{\lambda(z)}{\lambda(z) - 1}, \quad \lambda\left( - \frac{1}{z} \right) = 1 - \lambda(z).
\]
Write
\begin{equation}
    \label{eqn:lambdaS_def}
    \lambda_S(z) := \lambda\left(-\frac{1}{z}\right) = 1 - \lambda(z).
\end{equation}
Since $\lambda$ and $\lambda_S$ are nonvanishing on the simply connected domain $\mathbb{H}$, they admit holomorphic logarithms there; we denote by $\cL(z)$ and $\cL_S(z)$ the branches determined by the expansions
\begin{align}
    \cL(z) &= \pi i z + 4 \log 2 + \sum_{k \ge 1} (-1)^k \frac{r_4(k)}{k} q^{\frac{k}{2}}, \label{eqn:L_qexp} \\
    \cL_S(z) &= -16 \sum_{k \geq 0} \frac{\sigma_1(2k + 1)}{2k + 1}q^{k + \frac{1}{2}}. \label{eqn:LS_qexp}
\end{align}

\subsection{Serre derivative}

For an integer $k$, the Serre derivative $\partial_{k}$ is defined by
\[
    \partial_{k}F := F' - \frac{k}{12} E_2 F,
\]
for any quasimodular form $F$.
$\partial_{k}F$ is \emph{a priori} a quasimodular form of weight $w + 2$ and depth at most $s+1$ when $F \in \QM_w^s$. However, Kaneko and Koike \cite[Proposition 3.3]{kaneko2006extremal} proved that $\partial_{w-s}$ preserves the space of quasimodular forms of depth at most $s$.
The Serre derivative is equivariant under the $\mathrm{SL}_2(\mathbb{Z})$-action in the sense that
\[
\partial_k (F|_k \gamma) = (\partial_k F)|_{k+2}\gamma, \quad \forall \gamma \in \mathrm{SL}_2(\mathbb{Z}).
\]
The Serre derivative satisfies the product rule
\[
\partial_{w_1 + w_2} (FG) = (\partial_{w_1}F)G + F(\partial_{w_2}G).
\]
We denote the $r$-fold Serre derivative by
\[
\partial_k^r F := \partial_{k + 2(r-1)} \partial_{k + 2(r-2)} \cdots \partial_{k} F, \quad \partial_k^0 F := F.
\]
The Ramanujan identities \eqref{eqn:ramanujan} can be written as
\begin{equation}
    \label{eqn:ramanujan_serre}
    \partial_1 E_2 = - \frac{1}{12} E_4, \quad \partial_4 E_4 = -\frac{1}{3} E_6, \quad \partial_6 E_6 = - \frac{1}{2} E_4^2.
\end{equation}

\subsection{Sign uncertainty principle}

Bourgain--Clozel--Kahane's sign uncertainty principle has been studied by several authors.
In dimension $d=1$, the best currently known lower and upper bounds are
\[
0.45 \le \rA_+(1) \le 0.5671.
\]
The lower bound was proved using rearrangement inequalities motivated by optimal transport \cite[Theorem 1]{gonccalves2017hermite}.
The upper bound was established by Google DeepMind's autonomous agent \texttt{AlphaEvolve} \cite[Appendix B.4]{alphaevolve}, improving the earlier bounds of $0.59355$ from \cite[Theorem 1]{gonccalves2017hermite} and $0.572990$ from the numerical construction of Cohn--Gon{\c{c}}alves \cite[Table 2]{cohn2019optimal}.
All the upper bounds were obtained by optimizing over functions of the form (polynomial) $\times$ (Gaussian).

Gon{\c{c}}alves, Oliveira e Silva, and Steinerberger proved that the infimum defining $\rA_+(d)$ is attained by a nonzero self-Fourier function \cite{gonccalves2017hermite}. Consequently, $\rA_+(d)$ is the minimum of $r(g)$ over all nonzero $g \in L^{1}(\mathbb{R}^{d})$ such that $\what{g} = g$, $g(0) = 0$, and $g$ is eventually nonnegative \cite[Problem 1.1]{cohn2019optimal}.
The exact value of $\rA_+(d)$ is unknown in general, but Cohn and Gon{\c{c}}alves found the exact value when $d = 12$.
\begin{theorem}[{Cohn--Gon{\c{c}}alves \cite[Theorem 1.2]{cohn2019optimal}}]
$\rA_+(12) = \sqrt{2}$.
\end{theorem}
The lower bound $\rA_+(12) \ge \sqrt{2}$ follows from the Poisson-like summation formula associated with the Eisenstein series $E_6$.
The upper bound $\rA_+(12) \le \sqrt{2}$ is based on the explicit construction of an optimal function. For $\|\bx\|>\sqrt{2}$, this function is given by
\[
f(\bx) = \sin^2\left(\frac{\pi\|\bx\|^2}{2}\right) \int_{0}^{\infty} \frac{\Theta_4(it)^{12}(2\Theta_2(it)^4 + \Theta_4(it)^4)}{\Delta(it)} e^{-\pi \|\bx\|^2 t} \dd t,
\]
where $\Theta_2$ and $\Theta_4$ are Jacobi theta functions:
\begin{equation}
    \label{eqn:jacobi_theta}
    \Theta_2(z) = \sum_{n \in \bZ} e^{\pi i (n + \frac{1}{2})^2 z}, \quad \Theta_4(z) = \sum_{n \in \bZ} (-1)^{n} e^{\pi i n^2 z}.
\end{equation}
Replacing the condition $\what{g} = g$ by $\what{g} = -g$ in the characterization above defines an analogous quantity $\rA_-(d)$, and Cohn and Gon{\c{c}}alves also studied both constants numerically.
In particular, they considered functions of the form $p(2\pi \|\bx\|^2) e^{-\pi \|\bx\|^2}$, where $p$ is a polynomial, and used a Laguerre basis to impose the Fourier-eigenfunction condition.
This construction gives better bounds than \eqref{eqn:bckupper} for dimensions $d \le 32$ (Figure \ref{fig:bounds_Aplus}). However, when the degree of $p$ is sublinear in $d$, the fundamental limit of this construction has the same order of growth as \eqref{eqn:bckupper} \cite{cohn2024sign}.
Both $\rA_+(d)$ and $\rA_-(d)$ grow on the order of $\sqrt{d}$.
Cohn--Gon{\c{c}}alves \cite[Conjecture 1.5]{cohn2019optimal} and Afkhami-Jeddi--Cohn--Hartman--de Laat--Tajdini \cite[(3.5)]{afkhami2020high} conjectured that $\lim_{d \to \infty} \rA_{\pm}(d)/\sqrt{d} = 1/\pi$. This was recently proved by an internal OpenAI model \cite[Chapter 1, Theorem 1.2]{openai10}, and in particular yields \eqref{eqn:Apos_limit}.
However, the proof is not effective and the $o(1)$ term is not explicit.

\subsection{Modular linear differential operators and equations}

Modular linear differential operators (MLDOs) are differential operators that preserve modularity.
More precisely, a differential operator $L$ is an MLDO of weight $K$ and type $(k,k+K)$ on $\Gamma$ if it is a finite-order linear differential operator with holomorphic coefficients such that
\[
    L(f|_k \gamma) = (Lf)|_{k+K} \gamma \quad \text{for all } \gamma \in \Gamma.
\]
For example, the Serre derivative $\partial_k$ is an MLDO of weight $2$ and type $(k,k+2)$ on $\Gamma = \SL_2(\bZ)$.
Another important example comes from the study of supersingular $j$-invariants of elliptic curves \cite{kaneko1998supersingular}. In that setting, solutions of the modular linear differential equation (MLDE)
\[
    L_{2,k} F_k := \left(D^2 - \frac{k+1}{6} E_2 D + \frac{k(k+1)}{12}E_2'\right) F_k = \left(\partial_k^2 - \frac{k(k+2)}{144}E_4\right) F_k = 0,
\]
are related to the reduction modulo $p$ of the supersingular polynomial $ss_p(X)$ for a prime $p$.
In general, for $\alpha \in \bR$, the operator
\begin{equation}
    \label{eqn:KZ2_def}
    L_{2,k}^{\alpha} := L_{2,k} + \alpha E_4 = D^2 - \frac{k+1}{6} E_2 D + \frac{k(k+1)}{12}E_2' + \alpha E_4 = \partial_k^2 - \left(\frac{k(k+2)}{144} - \alpha\right) E_4,
\end{equation}
is an MLDO of weight $4$ and type $(k,k+4)$ on $\SL_2(\bZ)$.
When $k \equiv 0 \pmod{6}$, the depth 1, weight $k$ extremal quasimodular form $X_{k,1}$ satisfies the modular linear differential equation $L_{2,k-1} X_{k,1} = 0$ \cite{kaneko2006extremal,grabner2020quasimodular}.
Kaneko, Nagatomo, and Sakai \cite{kaneko2017third} studied the third-order analogue of the Kaneko--Zagier operator, which is an MLDO of weight $6$ and type $(k,k+6)$ on $\SL_2(\bZ)$ of the form
\begin{align}
    L_{3,k}^{(\alpha,\beta)} &:= D^3 - \frac{k+2}{4} E_2 D^2 + \left(\frac{(k+1)(k+2)}{4}E_2' + \alpha E_4\right) D - \left(\frac{k(k+1)(k+2)}{24}E_2'' + \frac{k\alpha}{4} E_4' - \beta E_6\right) \label{eqn:KZ3_def} \\
    &= \partial_k^3 + \left(\alpha - \frac{3k^2 + 12k + 8}{144}\right) E_4 \partial_k + \left(\beta + \frac{k\alpha}{12} - \frac{k^2(k+3)}{864}\right) E_6, \label{eqn:KZ3_def_serre}
\end{align}
where $\alpha, \beta \in \bR$ are parameters.
We will write $L_{3,k} := L_{3,k}^{(0,0)}$.
Nagatomo, Sakai, and Zagier showed that all MLDOs can be expressed in terms of (generalized) Rankin--Cohen brackets and Serre derivatives \cite{nagatomo2024modular}.

The following lemma gives a formula for the Fourier coefficients of the second-order Kaneko--Zagier operator applied to a quasimodular form.
\begin{lemma}
    \label{lem:KZ2_coeff}
    Let $G = \sum_{n \ge 0} a_n q^n$ be a quasimodular form, and let $n \ge 0$.
    For $k \in \bZ$ and $\alpha \in \bR$, the $n$-th Fourier coefficient of $L_{2,k}^{\alpha} G$ is
    \begin{equation}
        \label{eqn:KZ2_coeff}
        [q^n] (L_{2,k}^{\alpha} G) = \kappa_{2,k}^{\alpha}(n) a_n + \sum_{j=0}^{n-1} K_{2,k}^{\alpha}(n,j) a_j,
    \end{equation}
    where
    \begin{align}
        \kappa_{2,k}^{\alpha}(n) &= n^2 - \frac{k+1}{6}n + \alpha, \label{eqn:kappa2} \\
        K_{2,k}^{\alpha}(n,j) &= 2(k+1)(2j-k(n-j)) \sigma_1(n-j) + 240 \alpha \sigma_3(n-j). \label{eqn:K2}
    \end{align}
\end{lemma}
\begin{proof}
    By \eqref{eqn:KZ2_def}, the $n$-th coefficient of $L_{2,k}^{\alpha}G$ is
    \begin{align*}
        &n^2 a_n - \frac{k+1}{6}\left(na_n - 24 \sum_{j=0}^{n-1} j\sigma_1(n-j)a_j\right) - \frac{k(k+1)}{12} \cdot 24 \sum_{j=0}^{n-1} (n-j) \sigma_1(n-j) a_j + \alpha a_n + 240\alpha \sum_{j=0}^{n-1} \sigma_3(n-j) a_j \\
        &= \left(n^2 - \frac{k+1}{6}n + \alpha\right) a_n + \sum_{j=0}^{n-1} \left(2(k+1)(2j-k(n-j)) \sigma_1(n-j) + 240\alpha \sigma_3(n-j)\right) a_j,
    \end{align*}
    which implies \eqref{eqn:kappa2} and \eqref{eqn:K2}.
\end{proof}

One can obtain a similar formula for the third-order Kaneko--Zagier operator.
\begin{lemma}
    \label{lem:KZ3_coeff}
    Let $G = \sum_{n \ge 0} a_n q^n$ be a quasimodular form, and let $n \ge 0$.
    For $k \in \bZ$ and $\alpha,\beta \in \bR$, the $n$-th Fourier coefficient of $L_{3,k}^{(\alpha,\beta)}G$ is
    \begin{equation}
        [q^n] (L_{3,k}^{(\alpha,\beta)}G) = \kappa_{3,k}^{(\alpha,\beta)}(n) a_n + \sum_{j=0}^{n-1} K_{3,k}^{(\alpha,\beta)}(n,j) a_j,
    \end{equation}
    where
    \begin{align}
        \kappa_{3,k}^{(\alpha,\beta)}(n) &:= n^3 - \frac{k+2}{4}n^2 + \alpha n + \beta, \label{eqn:KZ3_kappa} \\
        K_{3,k}^{(\alpha,\beta)}(n,j) &:= (k+2)(6j^2 - 6(k+1)(n-j)j + k(k+1)(n-j)^2) \sigma_1(n-j) \nonumber \\
        &\quad + 60\alpha(4j - k(n-j)) \sigma_3(n-j) - 504\beta \sigma_5(n-j). \label{eqn:KZ3_K}
    \end{align}
\end{lemma}
\begin{proof}
    By \eqref{eqn:KZ3_def}, the $n$-th coefficient of $L_{3,k}^{(\alpha,\beta)}G$ is
    \begin{align*}
        &n^3 a_n - \frac{k+2}{4}\left(n^2 a_n - 24 \sum_{j=0}^{n-1} \sigma_1(n-j)j^2 a_j\right) \\
        &\quad - \frac{(k+1)(k+2)}{4}\cdot 24 \sum_{j=0}^{n-1} (n-j)\sigma_1(n-j) ja_j
         + \alpha \left(na_n + 240 \sum_{j=0}^{n-1} \sigma_3(n-j) ja_j \right) \\
        &\quad + \frac{k(k+1)(k+2)}{24} \cdot 24\sum_{j=0}^{n-1} (n-j)^2 \sigma_1(n-j) a_j
        - \frac{k\alpha}{4} \cdot240 \sum_{j=0}^{n-1} (n-j) \sigma_3(n-j) a_j \\
        &\quad + \beta \left(a_n - 504 \sum_{j=0}^{n-1} \sigma_5(n-j)a_j\right),
    \end{align*}
    which implies \eqref{eqn:KZ3_kappa} and \eqref{eqn:KZ3_K}.
\end{proof}

The second- and third-order Kaneko--Zagier operators satisfy the following intertwining criterion.
\begin{lemma}
    \label{lem:intertwine}
    Let $k \in \bZ$ and let
    $\alpha,\beta,\gamma,\alpha',\beta',\gamma' \in \bR$.
    Define the shifted parameters
    \begin{align*}
        A &:= \alpha - \frac{3k^2+36k+104}{144}, &
        A' &:= \alpha' - \frac{3k^2+12k+8}{144}, \\
        B &:= \beta + \frac{k+4}{12}\alpha - \frac{(k+4)^2(k+7)}{864}, &
        B' &:= \beta' + \frac{k}{12}\alpha' - \frac{k^2(k+3)}{864}, \\
        C &:= \gamma - \frac{k(k+2)}{144}, &
        C' &:= \gamma' - \frac{(k+6)(k+8)}{144}.
    \end{align*}
    Then the intertwining relation
    \begin{equation}
        \label{eqn:intertwine}
        L_{3,k+4}^{(\alpha,\beta)} L_{2,k}^{\gamma} = L_{2,k+6}^{\gamma'} L_{3,k}^{(\alpha',\beta')}
    \end{equation}
    holds if
    \begin{equation}
        \label{eqn:intertwine_constraints}
        A+C = A'+C', \quad
        B-C = B'-\frac{2}{3}A', \quad
        C\left(A+\frac{1}{2}\right) = A'\left(C'+\frac{1}{6}\right)-B', \quad
        C\left(B-\frac{A}{3}-\frac{1}{9}\right) = B'\left(C'+\frac{1}{3}\right).
    \end{equation}
\end{lemma}
\begin{proof}
    By the definitions above and \eqref{eqn:KZ3_def_serre}, the four operators in \eqref{eqn:intertwine} have the Serre-derivative forms
    \begin{align*}
        L_{2,k}^{\gamma} &= \partial_k^2 + C E_4,
        &L_{3,k+4}^{(\alpha,\beta)} &= \partial_{k+4}^3 + A E_4\partial_{k+4} + B E_6, \\
        L_{2,k+6}^{\gamma'} &= \partial_{k+6}^2 + C' E_4,
        &L_{3,k}^{(\alpha',\beta')} &= \partial_k^3 + A' E_4\partial_k + B' E_6.
    \end{align*}
    Repeatedly applying the product rule together with \eqref{eqn:ramanujan_serre} gives the following normal forms for the two sides:
    \begin{align}
        L_{3,k+4}^{(\alpha,\beta)} L_{2,k}^{\gamma}
        &= \partial_k^5 +(A+C)E_4\partial_k^3 +(B-C)E_6\partial_k^2 + C\left(A+\frac{1}{2}\right)E_4^2\partial_k
        + C\left(B-\frac{A}{3}-\frac{1}{9}\right)E_4E_6, \label{eqn:intertwine_lhs} \\
        L_{2,k+6}^{\gamma'} L_{3,k}^{(\alpha',\beta')}
        &= \partial_k^5 +(A'+C')E_4\partial_k^3
        + \left(B'-\frac{2A'}{3}\right)E_6\partial_k^2 + \left(A'\left(C'+\frac{1}{6}\right)-B'\right)E_4^2\partial_k
        + B'\left(C'+\frac{1}{3}\right)E_4E_6. \label{eqn:intertwine_rhs}
    \end{align}
    Now the intertwining relation follows from the constraint equations \eqref{eqn:intertwine_constraints}.
\end{proof}

\section{Positive quasimodular forms}

In \cite{lee2024algebraic}, the author defined \emph{positive and completely positive quasimodular forms} and studied their basic properties.
In particular, that paper describes how positivity and complete positivity behave under (Serre) derivatives and antiderivatives, and gives a new proof of the modular form inequalities and of Kaneko--Koike's conjecture in the case of depth 1.
In this section, we briefly recall these results and prove a \emph{weak version} of Kaneko--Koike's conjecture in the depth $2$ case.

\subsection{Positivity and derivatives}

We recall the definition of positive and completely positive quasimodular forms from the author's previous work \cite{lee2024algebraic}.

\begin{definition}
\label{def:posqm}
A (nonzero) quasimodular form $F \in \QM_{w}^{s}$ is \emph{positive} if it takes positive real values on the positive imaginary axis, that is, if $F(it) > 0$ for all $t > 0$.
Moreover, we call $F$ \emph{completely positive} if it has real and nonnegative Fourier coefficients, i.e. $F(z) = \sum_{n \geq n_0} a_n q^n$ with $a_n \geq 0$ for all $n \geq n_0$.
We denote the sets of positive and completely positive quasimodular forms of weight $w$ and depth $s$ by $\QM_{w}^{s, +}$ and $\QM_{w}^{s, ++}$, respectively.
\end{definition}

The following proposition summarizes the behavior of positivity and complete positivity under derivatives and antiderivatives, including their Serre analogues, as established in \cite[Section 3]{lee2024algebraic}.

\begin{proposition}[Lee \cite{lee2024algebraic}]
\label{prop:derpos}
Let $F = \sum_{n \geq n_0} a_n q^n \in \QM_{w}^{s}$ and $F' = \sum_{n \geq n_0} na_n q^n \in \QM_{w+2}^{s + 1}$.
\begin{enumerate}
    \item If $F$ is a cusp form, $F \in \QM_{w}^{s, ++}$ if and only if $F' \in \QM_{w+2}^{s + 1, ++}$.
    \item If $F$ is a cusp form, then $F' \in \QM_{w +2}^{s +1, +}$ implies $F \in \QM_{w}^{s, +}$.
    \item If $F$ is a cusp form, then $F$ is completely positive if and only if all its derivatives are positive.
    \item If $\partial_k F \in \QM_{w+2}^{s+1, +}$ for some $k$ and $F(it_0) > 0$ for some $t_0 > 0$, then $F(it) > 0$ for all $0 < t \le t_0$.
    In particular, if $\partial_k F \in \QM_{w+2}^{s+1,+}$ and $F(it) > 0$ for sufficiently large $t$, then $F \in \QM_{w}^{s, +}$.
    \item If $F \in \QM_{w}^{s,++}$ and $n_0 \geq k / 12 \geq 0$, then $\partial_k F \in \QM_{w+2}^{s+1, ++}$.
\end{enumerate}
\end{proposition}

\subsection{Extremal quasimodular forms and positivity}

\subsubsection{Extremal quasimodular forms}

In \cite{kaneko2006extremal}, Kaneko and Koike defined and studied \emph{extremal quasimodular forms}, which are the quasimodular forms of a given depth with the maximum possible order of vanishing at infinity.
In other words, for a given weight $w$ and depth $s$, a quasimodular form $f \in \QM_{w}^{s} \setminus \QM_{w}^{s-1}$ is \emph{extremal} if, for $m = \dim_\mathbb{C} \QM_{w}^{s}$, the first $m$ Fourier coefficients of $f = \sum_{n \geq 0} a_n q^n$ are
\[
    a_0 = a_1 = \cdots = a_{m-2} = 0, \qquad a_{m-1} \neq 0.
\]
They conjectured the existence and uniqueness (up to a nonzero scalar) of extremal forms for each even weight $w$ and depth $s$ satisfying $0 \leq s \leq w/2$ and $s \neq w/2 - 1$, and gave recursively defined examples in depths $1$ and $2$ that satisfy certain differential equations.
Pellarin \cite{pellarin2020extremal} established the conjecture for $s \leq 4$, and
Grabner \cite{grabner2020quasimodular} extended Kaneko--Koike's result, constructing differential equations and recurrence relations satisfied by extremal quasimodular forms of depth $\leq 4$ using vector-valued quasimodular forms.
Kaneko and Koike also conjectured that the Fourier coefficients of extremal forms of depth $\leq 4$ are all positive \cite[Conjecture 2]{kaneko2006extremal}, and Grabner \cite{grabner2022asymptotic} proved the conjecture for \emph{all but finitely many coefficients}.
The proof uses the explicit bounds of Jenkins and Rouse \cite{jenkins2011bounds} and coefficient asymptotics obtained from Deligne's bound.
The conjecture was proved in full for depth 1 in \cite[Corollary 4.4]{lee2024algebraic}.

\subsubsection{Positivity of extremal quasimodular forms of depth $2$}

For even $w \geq 4$ with $w \equiv 0 \pmod{4}$, the normalized depth 2 extremal forms $X_{w, 2}$ satisfy the following recurrence relations \cite{grabner2020quasimodular}\footnote{There is a minor error in \cite{grabner2020quasimodular}: one must replace $w^2$ in the numerator by $(w+4)^2$ in order for $X_{w+4,2}$ to be normalized. We make this correction in \eqref{eqn:d2eq1}.}: $X_{4, 2} = \frac{E_4 - E_2^2}{288}$ and
\begin{align}
    X_{w+4, 2} &= \frac{3(w+4)^2}{16(w+1)(w+2)^{2}(w+3)} \left(\frac{w(w+1)}{36} E_4 X_{w, 2} - \partial_{w-2}^{2}X_{w, 2}\right), \label{eqn:d2eq1} \\
    X_{w+2, 2} &= \frac{6}{w + 1} \partial_{w-2}X_{w, 2} \label{eqn:d2eq2} \\
    &= \frac{3w^2}{16(w^2 - 1)(w-6)^2} \left(\frac{(w-4)(w-5)}{36} E_4 X_{w-2,2} - \partial_{w-4}^{2}X_{w-2, 2}\right). \label{eqn:d2eq3}
\end{align}
Here \eqref{eqn:d2eq1} and \eqref{eqn:d2eq2} hold for all $w \equiv 0 \pmod{4}$ with $w \geq 4$, whereas \eqref{eqn:d2eq3} holds for $w \equiv 0 \pmod{4}$ with $w \geq 12$ (note that $X_{6,2}$ does not exist, so \eqref{eqn:d2eq3} is not available for $w = 8$).
The vanishing order of $X_{w, 2}$ at the cusp is $\lfloor \frac{w}{4} \rfloor$.
Also, when $w \equiv 0 \pmod{4}$, $X_{w,2}$ is a solution of the differential equation
\begin{align}
L_{3,w-2} X_{w, 2} &=  X_{w, 2}''' - \frac{w}{4} E_2 X_{w, 2}'' + \frac{w(w-1)}{4} E_{2}'X_{w, 2}' - \frac{w(w-1)(w-2)}{24} E_2'' X_{w, 2} \label{eqn:d2ode2} \\
&= \partial_{w-2}^{3}X_{w, 2} - \frac{3w^2 - 4}{144} E_4 \partial_{w-2}X_{w, 2} - \frac{(w-2)^{2}(w+1)}{864} E_6 X_{w, 2} = 0, \label{eqn:d2ode3}
\end{align}
where $L_{3,w-2} = L_{3,w-2}^{(0,0)}$ is the third-order Kaneko--Zagier operator defined in \eqref{eqn:KZ3_def}.

They also satisfy the following recurrence relation:
\begin{proposition}
\label{prop:d2eq}
For each multiple $w \geq 12$ of $4$, we have
\begin{equation}
\label{eqn:d2eq4}
    X_{w+2, 2} = \frac{w^{2}}{768(w - 1)(w+1)}(E_4 X_{w-2, 2}  - E_6 X_{w-4, 2}).
\end{equation}
\end{proposition}
\begin{proof}
Applying $\partial_{w-2}$ to \eqref{eqn:d2eq1}$_{w-4}$ and using \eqref{eqn:d2eq2}${}_{w}$ expresses $X_{w+2, 2}$ as a combination of $\partial_{w-2}(E_4 X_{w-4, 2})$ and $\partial_{w-6}^{3}X_{w-4, 2}$.
We then use \eqref{eqn:d2ode3}${}_{w-4}$ to eliminate the $\partial_{w-6}^{3}X_{w-4, 2}$ term.
This proves \eqref{eqn:d2eq4}${}_{w}$.
\end{proof}

For the admissible even weights $w \leq 14$, complete positivity follows from exceptional identities \cite[Proposition 4.6]{lee2024algebraic}.
For general weights, we prove a weak version of the Kaneko--Koike conjecture, namely the positivity of $X_{w,2}$, which is sufficient for our purposes.
The key idea is to use Nakaya's hypergeometric identity \cite{nakaya2024determination} and equation \eqref{eqn:d2eq4}.
Let ${}_{3}F_{2}$ be the hypergeometric function
\begin{equation}
    {}_{3}F_{2}(a_1, a_2, a_3; b_1, b_2; t) = \sum_{n \ge 0} \frac{(a_1)_n (a_2)_n (a_3)_n}{(b_1)_n (b_2)_n} \frac{t^n}{n!}.
\end{equation}
Nakaya proved that $X_{w, 2}$ admits the following hypergeometric expressions.
\begin{theorem}[Nakaya {\cite[Proposition 6.1]{nakaya2024determination}}]
\label{thm:extd2nakaya}
The normalized extremal quasimodular forms $X_{w, 2}$ for even $w \geq 4$ with $w \neq 6$ can be expressed as hypergeometric series:
\begin{align}
    X_{4k, 2}(z) &= j(z)^{-k} E_4(z)^{\frac{2k - 1}{2}} \cdot {}_{3}F_{2}\left(\frac{4k + 1}{6}, \frac{4k+3}{6}, \frac{4k+5}{6}; k+1, k+1; \frac{1728}{j(z)}\right) \label{eqn:extd2hg1}, \\
    X_{4k+2, 2}(z) &= j(z)^{-k} E_4(z)^{\frac{2k - 3}{2}} E_6(z) \cdot {}_{3}F_{2}\left(\frac{4k+3}{6}, \frac{4k+5}{6}, \frac{4k+7}{6}; k+1, k+1; \frac{1728}{j(z)}\right). \label{eqn:extd2hg2}
\end{align}
Here $k \geq 1$ in \eqref{eqn:extd2hg1} and $k \geq 2$ in \eqref{eqn:extd2hg2}. The first identity holds for $z=it$ and $t \geq 1$. The second holds for $z=it$ and $t>1$ and extends continuously to $t=1$; at $t=1$, it is understood in this limiting sense.
\end{theorem}
\begin{proof}
For the sake of completeness, we provide the details of the proof, which are omitted in \cite{nakaya2024determination}.
When $w = 4k$, \eqref{eqn:extd2hg1} can be shown by proving that both sides satisfy the same differential equation \eqref{eqn:d2ode2} and have the same normalized leading term at the cusp, as suggested in \cite[Section 6.1]{nakaya2024determination}.
In particular, it follows from \eqref{eqn:d2ode2} that the function $g(x) = E_4(z)^{-\frac{w-2}{4}}X_{w, 2}(z)$, where $x = 1728 / j(z)$, satisfies the equation
\begin{align*}
    &x^2(1 - x) g'''(x) + x \left(-\frac{w-12}{4} + \frac{w-18}{4} x\right) g''(x) \\
    &\quad + \left(-\frac{w-4}{4} - \frac{3w^2 - 72w + 452}{144}x\right) g'(x) + \frac{(w-10)(w-6)(w-2)}{1728} g(x) = 0.
\end{align*}
For $w=4k$, direct substitution shows that the right-hand side of \eqref{eqn:extd2hg1} yields a solution of this equation. Its leading term at the cusp is $q^k$, matching the normalization of $X_{4k,2}$, so uniqueness of the local solution with this exponent proves \eqref{eqn:extd2hg1}.
This differential equation has singularities at $x = 0,1,\infty$, and the analytic continuation of the solution $g(x)$ from $x = 0$ to $x = 1$ corresponds to $t$ decreasing from $\infty$ to $1$ for $z = it$.

For $X_{w + 2, 2} = X_{4k+2,2}$, we use \eqref{eqn:extd2hg1} and \eqref{eqn:d2eq2}:
\begin{align*}
    \frac{w+1}{6} X_{w+2, 2} &= \partial_{w-2} X_{w, 2} = X_{w, 2}' - \frac{w-2}{12} E_2 X_{w, 2} \\
    &= (j(z)^{-\frac{w}{4}} E_4(z)^{\frac{w-2}{4}})' \cdot {}_{3}F_{2}\left(\frac{w+1}{6},\frac{w+3}{6},\frac{w+5}{6}; \frac{w}{4} + 1, \frac{w}{4} + 1; \frac{1728}{j(z)}\right) \\
    &\quad+ j(z)^{-\frac{w}{4}} E_4(z)^{\frac{w-2}{4}} \left(\frac{1728}{j(z)}\right)' \cdot \frac{\dd}{\dd x} {}_{3}F_{2}\left(\frac{w+1}{6},\frac{w+3}{6},\frac{w+5}{6}; \frac{w}{4} + 1, \frac{w}{4} + 1; x\right) \bigg|_{x = 1728 / j(z)} \\
    &\quad - \frac{w-2}{12} E_2(z) j(z)^{-\frac{w}{4}} E_4(z)^{\frac{w-2}{4}} \cdot {}_{3}F_{2}\left(\frac{w+1}{6},\frac{w+3}{6},\frac{w+5}{6}; \frac{w}{4} + 1, \frac{w}{4} + 1; \frac{1728}{j(z)}\right) \\
    &= \partial_{w-2} (j(z)^{-\frac{w}{4}} E_4(z)^{\frac{w-2}{4}}) \cdot {}_{3}F_{2}\left(\frac{w+1}{6},\frac{w+3}{6},\frac{w+5}{6}; \frac{w}{4} + 1, \frac{w}{4} + 1; \frac{1728}{j(z)}\right) \\
    &\quad + j(z)^{-\frac{w}{4}} E_4(z)^{\frac{w-2}{4}} \left(\frac{1728}{j(z)}\right)' \cdot \frac{\dd}{\dd x} {}_{3}F_{2}\left(\frac{w+1}{6},\frac{w+3}{6},\frac{w+5}{6}; \frac{w}{4} + 1, \frac{w}{4} + 1; x\right) \bigg|_{x = 1728 / j(z)} \\
    &= \frac{w+1}{6} j(z)^{-\frac{w}{4}} E_4(z)^{\frac{w-6}{4}} E_6(z) \cdot {}_{3}F_{2} \left(\frac{w+1}{6},\frac{w+3}{6},\frac{w+5}{6}; \frac{w}{4} + 1, \frac{w}{4} + 1; \frac{1728}{j(z)}\right) \\
    &\quad + j(z)^{-\frac{w}{4}} E_4(z)^{\frac{w-6}{4}} E_6(z) \cdot x\frac{\dd}{\dd x} {}_{3}F_{2}\left(\frac{w+1}{6},\frac{w+3}{6},\frac{w+5}{6}; \frac{w}{4} + 1, \frac{w}{4} + 1; x\right) \bigg|_{x = 1728 / j(z)} \\
    &= \frac{w+1}{6} j(z)^{-\frac{w}{4}} E_4(z)^{\frac{w-6}{4}} E_6(z) \cdot {}_{3}F_{2} \left(\frac{w+7}{6},\frac{w+3}{6},\frac{w+5}{6}; \frac{w}{4} + 1, \frac{w}{4} + 1; \frac{1728}{j(z)}\right) \\
    &= \frac{w+1}{6} j(z)^{-\frac{w}{4}} E_4(z)^{\frac{w-6}{4}} E_6(z) \cdot {}_{3}F_{2} \left(\frac{w+3}{6},\frac{w+5}{6},\frac{w+7}{6}; \frac{w}{4} + 1, \frac{w}{4} + 1; \frac{1728}{j(z)}\right).
\end{align*}
The penultimate equality uses the identity
\[
\left(1 + \frac{x}{a_1} \frac{\dd}{\dd x}\right) {}_{3}F_{2}(a_1, a_2, a_3; b_1, b_2; x) = {}_{3}F_{2}(a_1 + 1, a_2, a_3; b_1, b_2; x).
\]
For $t>1$, we have $0<1728/j(it)<1$. The first series also converges at $t=1$ because $b_1 + b_2 - a_1 - a_2 - a_3 = (k + 1) + (k + 1) - \frac{4k+1}{6} - \frac{4k+3}{6} - \frac{4k+5}{6} = \frac{1}{2} > 0$, while continuity of $X_{4k+2,2}$ at $i$ gives the claimed limiting interpretation of the second identity.
\end{proof}

\begin{theorem}
\label{thm:kkd2weak}
$X_{w, 2}$ is positive for every even $w \geq 4$ with $w \neq 6$.
\end{theorem}
\begin{proof}
We use induction on multiples $w$ of $4$, proving the result simultaneously for $X_{w,2}$ and $X_{w+2,2}$, and consider the cases $t > 1$ and $0 < t \le 1$ separately.
The admissible cases $w \leq 14$ follow from the complete positivity noted above and serve as the base cases.
Throughout, we use the standard facts that $E_4(it) > 0$ for all $t > 0$, that $E_6(it) > 0$ for $t > 1$, and that $E_6(i) = 0$. The transformation formula $E_6(i/t)=-t^6E_6(it)$ then gives $E_6(it)<0$ for $0<t<1$.
The case $t>1$ follows from \eqref{eqn:extd2hg1} and \eqref{eqn:extd2hg2}: the hypergeometric series have positive coefficients, and $E_4(it)$, $E_6(it)$, and $j(it)$ are positive on this interval.
Nonnegativity at $t=1$ follows by continuity.
We now consider $0 < t \le 1$.
Assume that $X_{w-2, 2}$ and $X_{w-4, 2}$ are positive for $0 < t \le 1$.
By \eqref{eqn:d2eq4} and $E_6(i) = 0$, we have
\[
    X_{w+2,2}(i) = \frac{w^2}{768(w-1)(w+1)} E_4(i) X_{w-2,2}(i)
\]
for all $w \equiv 0 \pmod{4}$ with $w \geq 12$.
From $E_4(i) > 0$ and $X_{10,2}(i) > 0$, we conclude that $X_{w+2,2}(i) > 0$ for all $w \equiv 0 \pmod{4}$ with $w \geq 12$.
Since $E_4(it) > 0$ and $E_6(it) \le 0$ for $0 < t \le 1$, \eqref{eqn:d2eq4} shows that $X_{w + 2, 2}$ is also positive on $0 < t \le 1$.
Finally, Proposition \ref{prop:derpos} (4) and \eqref{eqn:d2eq2} show that $X_{w, 2}$ is positive for $0 < t \le 1$.
\end{proof}

\begin{remark}
In \cite{kaneko2017third}, certain solutions of third-order Kaneko--Zagier MLDEs \cite{kaneko1998supersingular} are characterized.
In particular, solutions of \emph{(vacuum) character type} are characterized; these are solutions $f$ of weight $k$ for which the quotient $f / \eta^{2k}$ has nonnegative integral Fourier coefficients, implying the positivity of $f$.
However, our $X_{w, 2}$ are not of this type since the coefficients are not integral (after normalization).
\end{remark}

\section{Fourier eigenfunctions of Feigenbaum--Grabner--Hardin}

In \cite{feigenbaum2021eigenfunctions}, Feigenbaum, Grabner, and Hardin constructed families of Fourier eigenfunctions from modular forms in dimensions divisible by $4$.
In dimensions $8$, $12$, and $24$, their construction recovers the magic functions of Viazovska \cite{viazovska2017sphere}, Cohn--Gon{\c{c}}alves \cite{cohn2019optimal}, and Cohn--Kumar--Miller--Radchenko--Viazovska \cite{cohn2017sphere}, respectively.
Following Viazovska, they constructed the $(-1)^{d/4}$- and $(-1)^{d/4 + 1}$-eigenfunctions separately as Laplace transforms of certain modular forms.

In this section, we briefly review both constructions.
For the $(-1)^{d/4}$-eigenfunctions, we relate the associated family of quasimodular forms to the depth 2 extremal quasimodular forms of Kaneko and Koike \cite{kaneko2006extremal} and use Theorem~\ref{thm:kkd2weak} to prove its positivity; for the $(-1)^{d/4+1}$-eigenfunctions, we develop a parallel argument using a level-$2$ companion family.

\subsection{$(-1)^{d/4}$-eigenforms}
\label{sec:fgh_poseig}

For a dimension $d \equiv 0\pmod{4}$, the $(-1)^{d/4}$-eigenfunctions of the Fourier transform arise from quasimodular forms of level $\SL_2(\bZ)$ and depth $2$, which can be expressed as polynomials in $E_2$, $E_4$, and $E_6$ \cite[Theorem 3.2 and Propositions 5.1 and 5.3]{feigenbaum2021eigenfunctions}.
To simplify computations further, we normalize the forms $f_w$ in \cite{feigenbaum2021eigenfunctions} so that their first nonzero Fourier coefficients are all $1$.
The following theorem is the normalized version of their results.

\begin{theorem}[Feigenbaum--Grabner--Hardin \cite{feigenbaum2021eigenfunctions}, normalized]
\label{thm:fghposeig}
For even $w \geq 8$, define quasimodular forms $\{F_w\}_{w \geq 8}$ of weight $w$ and depth $2$ as
\begin{align*}
F_{8} &= \frac{1}{1728}(E_2^2 E_4 - 2E_2 E_6 + E_4^2) \\
F_{10} &= \frac{1}{1728}(- E_2^2 E_6 + 2 E_2 E_4^2 - E_4 E_6) \\
F_{12} &= \frac{1}{518400}(E_2^2 E_4^2 - 2 E_2 E_4 E_6 + E_6^2)\\
F_{14} &= \frac{1}{725760}(-E_2^2 E_4 E_6 + E_2 E_4^3 + E_2 E_6^2 - E_4^2 E_6)\\
F_{16} &= \frac{1}{3657830400}(49 E_2^2 E_4^3 - 25 E_2^2 E_6^2 - 48 E_2 E_4^2 E_6 - 25 E_4^4 + 49 E_4 E_6^2)\\
F_{18} &= \frac{1}{2874009600} (-12 E_2^2 E_4^2 E_6 + 5 E_2 E_4^4 + 19 E_2 E_4 E_6^2 - 5 E_4^3 E_6 - 7 E_6^3)
\end{align*}
and, for $w \equiv 0 \pmod{4}$ (with $w \geq 12$ in the first recurrence and $w \geq 8$ in the second),
\begin{align}
    F_{w + 2} &= \frac{3(w-8)(w-4)}{16(w-18)(w-7)(w-6)(w-5)} \left(\frac{(w-11)(w-10)}{36} E_{4} F_{w-2} - \partial_{w-4}^{2} F_{w-2}\right) \label{eqn:poseig2} \\
    F_{w + 4} &= \frac{3(w-4)w}{16(w-10)(w-5)(w-3)(w+2)} \left(\frac{(w-6)(w-5)}{36} E_{4} F_{w} - \partial_{w-2}^{2} F_{w} \right).\label{eqn:poseig4}
\end{align}
Then the vanishing order of $F_{w}$ at the cusp is $\lfloor \frac{w}{4}\rfloor - 1$.
For $w \equiv 0\pmod{4}$, these forms $F_w$ satisfy the third-order ordinary differential equation
\begin{equation}
    L_{3,w-2}^{(\frac{w-4}{4},0)} F_w = \partial_{w-2}^{3}F_w - \frac{3w^2 - 36w + 140}{144}E_4 \partial_{w-2}F_{w} - \frac{(w-14)(w-5)(w-2)}{864}E_6 F_w = 0, \label{eqn:poseigde}
\end{equation}
or equivalently,
\begin{equation}
    \label{eqn:poseigde2}
    F_w''' - \frac{w}{4} E_2 F_w'' + \left(\frac{w(w-1)}{4} E_2' + \frac{w-4}{4}E_4\right)F_w' - \left(\frac{w(w-1)(w-2)}{24} E_2'' + \frac{(w-2)(w-4)}{16} E_4'\right) F_w = 0.
\end{equation}
Let $d$ now be a positive integer divisible by $4$, and set $n_{d,+} = \lfloor (d + 4)/16\rfloor + 1$.
Let $w = w_{d, +} = 12 \lfloor (d+4) / 16\rfloor - d /2  + 16$.
For $\mathbf{x}\in\mathbb{R}^{d}$, define
\begin{align}
    \label{eqn:Mdpos}
    M_{d, +}(\mathbf{x}) = 4 \sin^{2} \left(\frac{\pi \|\mathbf{x}\|^{2}}{2}\right) \int_{0}^{\infty} \frac{t^{2-w}F_{w}(i/t)}{\Delta(it)^{n_{d,+}}} e^{-\pi \|\mathbf{x}\|^{2} t} \dd t
\end{align}
Then this function satisfies (here we abuse notation by writing $M_{d,+}(\mathbf{x}) = M_{d,+}(\|\mathbf{x}\|)$)
\begin{align*}
    \widehat{M_{d,+}}(\mathbf{x}) &= (-1)^{d/4} M_{d,+}(\mathbf{x}) \quad \forall \mathbf{x} \in \mathbb{R}^{d}, \\
    M_{d,+}(\sqrt{2 n_{d,+}}) &= 0\quad\text{and}\quad M_{d,+}'(\sqrt{2n_{d,+}}) \neq 0, \\
    M_{d,+}(\sqrt{2m}) &= M_{d,+}'(\sqrt{2m}) = 0 \quad\forall m > n_{d,+}, m \in \mathbb{Z}.
\end{align*}
The integral \eqref{eqn:Mdpos} converges when $\|\bx\| > \sqrt{2n_{d,+}}$ and admits an analytic continuation to the origin (see Section \ref{subsec:nonposorig_pos}).
\end{theorem}

\begin{proof}
    The proof can be found in \cite{feigenbaum2021eigenfunctions}, except for the normalization claim.
    We first consider $w \equiv 0 \pmod{4}$.
    Let $F_w = \sum_{n \ge \frac{w}{4} - 1} b_n^{(w)} q^n$.
    The MLDE is \eqref{eqn:poseigde}, and its expression in terms of ordinary derivatives is \eqref{eqn:poseigde2}; the corresponding third-order Kaneko--Zagier operator is $L_{3,w-2}^{((w-4)/4, 0)}$.
    By applying Lemma \ref{lem:KZ3_coeff} to $F_{w}$ with $n = \frac{w}{4}$, we get
    \[
        0 = \frac{(w-4)w}{16} b_{\frac{w}{4}}^{(w)} - \frac{w^3 - 12w^2 + 224w - 960}{8} b_{\frac{w}{4}-1}^{(w)}
        \quad\Leftrightarrow\quad
        \frac{b_{\frac{w}{4}}^{(w)}}{b_{\frac{w}{4}-1}^{(w)}} = \frac{2(w^3 - 12w^2 + 224w - 960)}{w(w-4)}.
    \]
    We use induction on such $w$, the base case $w = 8$ being immediate from the expression for $F_8$ above.
    Assume that $b_{\frac{w}{4}-1}^{(w)} = 1$, so that
    \begin{equation}
        \label{eqn:F_qexp}
        F_w = q^{\frac{w}{4} - 1} + \frac{2(w^3 - 12w^2 + 224w - 960)}{w(w-4)} q^{\frac{w}{4}} + \cdots.
    \end{equation}
    Then direct computation with \eqref{eqn:poseig4} shows that $b_{\frac{w}{4}-1}^{(w+4)} = 0$ and $b_{\frac{w}{4}}^{(w+4)} = 1$.
    The case $w \equiv 2 \pmod{4}$ is proved similarly, starting from $w=10$ and using the corresponding MLDE \cite[eq. (5$\cdot$20)]{feigenbaum2021eigenfunctions} together with \eqref{eqn:poseig2}.
\end{proof}

From now on, we assume $d \equiv 0 \pmod{8}$, so that $(-1)^{d/4} = 1$ and $M_{d,+}$ is a $+1$-eigenfunction of the Fourier transform.
The corresponding weight $w_{d,+}$ and $n_{d,+}$ are
\begin{align}
    w_{d,+} = \begin{cases}
        \frac{d}{4} + 16 & d \equiv 0 \pmod{16} \\ \frac{d}{4} + 10 & d \equiv 8 \pmod{16}
    \end{cases}, \quad n_{d,+} = \begin{cases}
        \frac{d}{16} + 1 = \frac{w}{4} - 3 & d \equiv 0 \pmod{16} \\ \frac{d-8}{16} + 1 = \frac{w}{4} - 2  & d \equiv 8 \pmod{16}
    \end{cases}.
\end{align}
Note that $n_{d,+} = \left\lfloor \frac{d}{16} \right\rfloor + 1$ when $d \equiv 0 \pmod{8}$.

\subsubsection{Positivity of $F_w$}
\label{subsec:posFw}

In this section, we prove that the forms $F_w$ are positive for every even $w \geq 8$.
The proof is based on Propositions \ref{prop:serrederposeig} and \ref{prop:ext2poseig}, which reduce the positivity of the forms $F_w$ to Theorem \ref{thm:kkd2weak}.
\begin{proposition}
\label{prop:serrederposeig}
For each multiple $w$ of $4$, \eqref{eqn:poseigeq1} holds when $w \geq 12$, and \eqref{eqn:poseigeq2} holds when $w \geq 8$:
\begin{align}
F_{w+2} &= \frac{(w-8)(w-4)}{768(w-7)(w-5)} (E_4 F_{w - 2} - E_6 F_{w - 4})   \label{eqn:poseigeq1} \\
\partial_{w - 2} F_{w} &= \frac{w - 5}{6} F_{w + 2}. \label{eqn:poseigeq2}
\end{align}
\end{proposition}
\begin{proof}
The proof is similar to that of Proposition \ref{prop:d2eq}.
We can directly check \eqref{eqn:poseigeq2} for $w = 8$ and both equations for $w = 12$.
For the induction step, let $w \geq 12$ and assume that \eqref{eqn:poseigeq1}${}_{w}$ and \eqref{eqn:poseigeq2}${}_{w}$ hold.
In \eqref{eqn:poseig2}${}_{w+4}$, we can write $\partial_{w}^{2}F_{w+2}$ as $\frac{6}{w-5}\partial_{w-2}^{3}F_{w}$ using \eqref{eqn:poseigeq2}${}_{w}$, and use \eqref{eqn:poseigde}${}_{w}$ (and \eqref{eqn:poseigeq2}${}_{w}$ again) to express it as a combination of $E_4 F_{w+2}$ and $E_6 F_{w}$, which proves \eqref{eqn:poseigeq1}${}_{w+4}$.
\eqref{eqn:poseigeq2}${}_{w+4}$ can be shown by applying $\partial_{w+2}$ to \eqref{eqn:poseig4}${}_{w}$ and using \eqref{eqn:poseigde}${}_{w}$ and \eqref{eqn:poseigeq1}${}_{w+4}$.
\end{proof}

For $w \equiv 0 \pmod{4}$, write the Fourier expansions of the normalized forms as
\begin{equation}
    \label{eqn:ab_notation}
    X_{w, 2} = \sum_{n \geq \frac{w}{4}} a_{n}^{(w)} q^{n}, \qquad F_{w} = \sum_{n \geq \frac{w}{4}-1} b_{n}^{(w)} q^{n},
\end{equation}
so that $a_{\frac{w}{4}}^{(w)} = b_{\frac{w}{4}-1}^{(w)} = 1$; the coefficients $b_{n}^{(w)}$ are those of \eqref{eqn:F_qexp}.
The second nonzero Fourier coefficient of $X_{w-4, 2}$ is given by the following lemma.

\begin{lemma}
\label{lem:secondcoeff}
For $w \geq 12$ and $w \equiv 0 \pmod{4}$, we have
\begin{equation}
    a_{\frac{w}{4}}^{(w-4)} = \frac{2(w-4)(w^2 + 4w - 48)}{w^2} \label{eqn:possecondcoeff}
\end{equation}
\end{lemma}
\begin{proof}
Taking the coefficient of $q^{w/4}$ in \eqref{eqn:d2eq1}${}_{w-4}$ shows that the $\frac{w}{4}$-th Fourier coefficient of $X_{w, 2}$ is
\begin{align*}
\frac{3w^2}{16(w-1)(w-2)^2 (w-3)} \left( - \frac{2w - 1}{6}a_{\frac{w}{4}}^{(w-4)} + \frac{18w^2 - 129w + 240}{3}\right)
\end{align*}
(and $(\frac{w}{4}-1)$-th Fourier coefficient vanishes).
However, since the forms $X_{w, 2}$ are all normalized, the above expression must equal $1$, and solving for $a_{\frac{w}{4}}^{(w-4)}$ gives \eqref{eqn:possecondcoeff}${}_{w}$.
\end{proof}

We can now express the forms $F_w$ in terms of depth 2 extremal quasimodular forms as follows.
\begin{proposition}
\label{prop:ext2poseig}    
For $w \geq 12$ and $w \equiv 0\pmod{4}$, we have
\begin{align}
    F_{w} &= - \frac{256(w-3)(w-2)(w-1)}{(w-4)w^2} X_{w, 2} + E_4 X_{w-4,2} \label{eqn:poseigeq3} \\
    F_{w+2} &= \frac{2(w - 3)}{3(w-4)} E_4 X_{w-2, 2} + \frac{w-6}{3(w-4)} E_6 X_{w-4, 2}. \label{eqn:poseigeq4}
\end{align}
\end{proposition}
\begin{proof}
\eqref{eqn:poseigeq4}${}_{w}$ follows from \eqref{eqn:poseigeq3}${}_{w}$ by taking $\partial_{w-2}$ on both sides of \eqref{eqn:poseigeq3}${}_{w}$ and simplifying with \eqref{eqn:poseigeq2}${}_{w-4}$ and \eqref{eqn:poseigeq2}${}_{w}$.
Hence it is enough to prove \eqref{eqn:poseigeq3}.
First, both $F_{w}$ and $X_{w-4, 2}$ are normalized and have vanishing order $\frac{w}{4} - 1$ at the cusp.
Hence $E_{4} X_{w-4, 2} - F_{w}$ is a weight-$w$, depth 2 quasimodular form with vanishing order $\frac{w}{4}$ at the cusp; that is, it is an extremal quasimodular form.
By uniqueness \cite{pellarin2020extremal}, it is a constant multiple of $X_{w, 2}$; hence there exists a constant $c_{w}$ such that $E_{4} X_{w-4, 2} - F_{w} = c_{w}X_{w, 2}$.
To compute the constant explicitly, we compare the second nonzero Fourier coefficients of $F_{w}$ and $X_{w-4, 2}$: by \eqref{eqn:ab_notation},
\begin{align*}
    E_{4}X_{w-4, 2} - F_{w} &= \left(1 + 240q + \cdots \right)\left(q^{\frac{w}{4} - 1} + a_{\frac{w}{4}}^{(w-4)} q^{\frac{w}{4}} + \cdots\right) - \left(q^{ \frac{w}{4}- 1} + b_{\frac{w}{4}}^{(w)}q^{\frac{w}{4}} + \cdots\right) \\
    &=\left(240 + a_{\frac{w}{4}}^{(w-4)} - b_{\frac{w}{4}}^{(w)}\right)q^{\frac{w}{4}} + \cdots,
\end{align*}
so that $c_{w} = 240 + a_{\frac{w}{4}}^{(w-4)} - b_{\frac{w}{4}}^{(w)}$.
Using the formula for $b_{\frac{w}{4}}^{(w)}$ from \eqref{eqn:F_qexp} and Lemma \ref{lem:secondcoeff},
\begin{align*}
c_{w} &= 240 + a_{\frac{w}{4}}^{(w-4)} - b_{\frac{w}{4}}^{(w)} \\
&= \frac{240w^2(w-4) + 2(w-4)^{2}(w^2 + 4w - 48) - 2w(w^3 - 12w^2 + 224w - 960)}{(w-4)w^2} \\
&= \frac{256(w-3)(w-2)(w-1)}{(w-4)w^2}
\end{align*}
and this completes the proof of \eqref{eqn:poseigeq3}${}_{w}$.
\end{proof}

We can now prove the positivity of $F_w$ for all even weights in the family.

\begin{theorem}
\label{thm:poseigpos}
For every even $w\geq 8$, $F_w \in \QM_{w}^{2, +}$.
\end{theorem}
\begin{proof}
We prove simultaneously by induction on multiples $w \geq 8$ of $4$ that $F_w$ and $F_{w+2}$ are positive.
Ramanujan's identities give
\[
    F_8 = \frac{E_4''}{240} = \sum_{n \ge 1} n^2 \sigma_3(n) q^n,
    \qquad
    F_{10} = -\frac{E_6''}{504} = \sum_{n \ge 1} n^2 \sigma_5(n) q^n.
\]
Thus $F_8$ and $F_{10}$ are completely positive, and hence positive.
For $w \geq 12$, assume inductively that $F_{w-4}$ and $F_{w-2}$ are positive.
The scalar prefactor in \eqref{eqn:poseigeq1}${}_{w}$ is positive. Hence the inequalities $E_4(it) > 0$ and $E_6(it) \leq 0$ imply that $F_{w+2}(it) > 0$ for $0 < t \leq 1$; strictness follows already from the term $E_4(it)F_{w-2}(it)$.
For $t>1$, both scalar coefficients in \eqref{eqn:poseigeq4}${}_{w}$ are positive, so Theorem \ref{thm:kkd2weak} and the inequalities $E_4(it)>0$ and $E_6(it)>0$ likewise give $F_{w+2}(it)>0$.
Thus $F_{w+2}$ is positive. Equation \eqref{eqn:poseigeq2}${}_{w}$ and Proposition \ref{prop:derpos} (4) then imply that $F_w$ is positive, completing the simultaneous induction.
\end{proof}

\begin{remark}
For $w\geq12$ with $w\equiv 0\pmod{4}$, we also have the following relation between $F_{w+2}$, $X_{w+2, 2}$, and $X_{w-2, 2}$, similar to \eqref{eqn:poseigeq3}:
\begin{equation}
\label{eqn:poseigeq5}
F_{w+2} = -\frac{256(w-6)(w-1)(w+1)}{(w-4)w^2} X_{w+2, 2} + E_4 X_{w-2, 2},
\end{equation}
which can be proved similarly.
\end{remark}

\subsubsection{Nonpositivity of $M_{d,+}(\mathbf{0})$}
\label{subsec:nonposorig_pos}

Write $F_w = A_w + E_2 B_{w - 2} + E_2^2 C_{w-4}$ for modular forms $A_w, B_{w-2}, C_{w-4}$ of weight $w, w-2, w-4$ respectively.
Since $w \equiv 0 \pmod{4}$,
\begin{align*}
    F_w\left(\frac{i}{t}\right) &= A_w\left(\frac{i}{t}\right) + E_2\left(\frac{i}{t}\right) B_{w-2}\left(\frac{i}{t}\right) + E_2^2\left(\frac{i}{t}\right) C_{w-4}\left(\frac{i}{t}\right) \\
    &= t^{w} A_w(it) - \left(-t^2 E_2(it) + \frac{6t}{\pi}\right) t^{w-2} B_{w-2}(it) + \left(t^4 E_2^2(it) - \frac{12t^3}{\pi} E_2(it) + \frac{36t^2}{\pi^2}\right) t^{w-4} C_{w-4}(it) \\
    &= t^w F_w(it) - \frac{6t^{w-1}}{\pi} (B_{w-2}(it) + 2 E_2(it) C_{w-4}(it)) + \frac{36t^{w-2}}{\pi^2} C_{w-4}(it)
\end{align*}
and the integrand of \eqref{eqn:Mdpos} can be expressed as
\begin{equation}
    \label{eqn:phi}
    \frac{t^{2-w} F_w\left(i / t\right)}{\Delta(it)^{n_{d,+}}} = t^2 \frac{F_w(it)}{\Delta(it)^{n_{d,+}}} - \frac{6t}{\pi} \frac{B_{w-2}(it) + 2 E_2(it) C_{w-4}(it)}{\Delta(it)^{n_{d,+}}} + \frac{36}{\pi^2} \frac{C_{w-4}(it)}{\Delta(it)^{n_{d,+}}}.
\end{equation}
Set
\begin{equation}
    \label{eqn:Ftilde}
    \widetilde{F}_{w-2} := B_{w-2} + 2 E_2 C_{w-4}.
\end{equation}
This is a quasimodular form of weight $w-2$ and depth 1, and it is a cusp form. Indeed, if we initially write $\widetilde{F}_{w-2}=\sum_{n\geq 0}\tilde a_{n,+}^{(w-2)}q^n$, then taking constant terms in Lemma \ref{lem:Ftilde_rec} below shows that $\tilde{a}_{0,+}^{(w+2)}$ is a multiple of $\tilde{a}_{0,+}^{(w-2)}$.
Since $\wtilde{F}_6 = \frac{5}{6} X_{6,1}$ is a cusp form, we conclude that $\tilde{a}_{0,+}^{(w-2)} = 0$ for all $w \equiv 0 \pmod{4}$, hence
\[
    \widetilde{F}_{w-2}=\sum_{n\geq 1}\tilde a_{n,+}^{(w-2)}q^n.
\]
Therefore the Fourier expansions of the three quotients in \eqref{eqn:phi} have the form
\begin{equation}
    \label{eqn:qexp_integrand}
    \frac{F_{w}}{\Delta^{n_{d,+}}} = \sum_{n \geq \frac{w}{4}-1-n_{d,+}} a_{n,+} q^n, \quad
    \frac{\widetilde F_{w-2}}{\Delta^{n_{d,+}}} = \sum_{n \ge 1-n_{d,+}} b_{n,+} q^n, \quad \frac{C_{w-4}}{\Delta^{n_{d,+}}} = \sum_{n \ge - n_{d,+}} c_{n,+} q^n.
\end{equation}
Since $\frac{w}{4}-1-n_{d,+} > 0$, the integral over the first term in \eqref{eqn:phi} converges for all $\bx \in \bR^d$.
However, there are poles from the other two terms, which make the integral \eqref{eqn:Mdpos} diverge near the origin.
Fortunately, we can analytically continue the integral \eqref{eqn:Mdpos} to the origin by following the argument in \cite[Proposition 2.1]{feigenbaum2021eigenfunctions}, which originates in \cite{viazovska2017sphere,cohn2017sphere}.
More precisely, let $\phi(t)$ be \eqref{eqn:phi} and define the truncation $\widetilde{\phi}(t)$ as
\begin{equation}
    \label{eqn:phitilde}
    \widetilde{\phi}(t) := \phi(t) - \left[-\frac{6t}{\pi} \sum_{1-n_{d,+} \le k \le 0} b_{k,+} e^{-2 \pi k t} + \frac{36}{\pi^2} \sum_{-n_{d,+} \le k \le 0} c_{k,+} e^{-2 \pi k t}\right]
\end{equation}
which has exponential decay as $t \to \infty$.
Then \eqref{eqn:Mdpos} can be rewritten as
\begin{align*}
    M_{d,+}(\bx) &= 4 \sin^{2} \left(\frac{\pi \|\mathbf{x}\|^{2}}{2}\right) \left[\int_{0}^{\infty} \widetilde{\phi}(t) e^{-\pi \|\mathbf{x}\|^{2} t} \dd t \right.
    \\
    &\left.\quad - \sum_{1-n_{d,+} \le k \le 0} \frac{6 b_{k,+}}{\pi} \int_{0}^{\infty} t e^{-\pi \|\mathbf{x}\|^{2} t} e^{-2 \pi k t} \dd t + \sum_{-n_{d,+} \le k \le 0} \frac{36 c_{k,+}}{\pi^2} \int_{0}^{\infty} e^{-\pi \|\mathbf{x}\|^{2} t} e^{-2 \pi k t} \dd t \right] \\
    &= 4 \sin^2\left(\frac{\pi \|\bx\|^2}{2}\right)\int_0^{\infty} \widetilde{\phi}(t) e^{-\pi \|\bx\|^2 t} \dd t \\
    &\quad + 4 \sin^2\left(\frac{\pi \|\bx\|^2}{2}\right) \left[ - \frac{6}{\pi}\sum_{1-n_{d,+} \le k \le 0} \frac{b_{k,+}}{\pi^2 (\|\bx\|^2 + 2k)^2} + \frac{36}{\pi^2} \sum_{-n_{d,+} \le k \le 0} \frac{c_{k,+}}{\pi (\|\bx\|^2 + 2 k)}\right]
\end{align*}
which converges for $\|\bx\|^2 > 2n_{d,+}$ and analytically continues to $\bx = \mathbf{0}$.
The only term contributing to the value at the origin is $b_{0,+}$, which gives
\begin{equation}
    \label{eqn:Mpos_orig}
    M_{d,+}(\mathbf{0}) = -\frac{6}{\pi} \cdot 4 \cdot \frac{\pi^2}{4} \cdot \frac{b_{0,+}}{\pi^2} = - \frac{6 b_{0,+}}{\pi}.
\end{equation}

From \eqref{eqn:Mpos_orig}, we have $M_{d,+}(\mathbf{0}) \leq 0$ if and only if $b_{0,+} \geq 0$.
Now, for each integer $k \ge 0$, define $p_k(n)$ by
\begin{equation}
    \label{eqn:pkn}
    \frac{1}{\prod_{m \ge 1} (1 - q^m)^k} = \sum_{n \ge 0} p_k(n) q^n.
\end{equation}
Thus $p_1(n) = p(n)$ is the number of partitions of $n$; moreover, $\Delta^{-m} = q^{-m} \sum_{n \geq 0} p_{24m}(n) q^n$.
By \eqref{eqn:qexp_integrand} and \eqref{eqn:Ftilde}, we can write $b_{0,+}$ as
\begin{equation}
    \label{eqn:bzero_pos}
    b_{0,+} = [q^{n_{d,+}}] \left(\sum_{n \geq 0} p_{24 n_{d,+}}(n) q^n\right) \left(\sum_{n \geq 1} \tilde{a}_{n,+}^{(w-2)} q^n \right) = \sum_{j=1}^{n_{d,+}} p_{24 n_{d,+}}(n_{d,+} - j) \tilde{a}_{j,+}^{(w-2)}.
\end{equation}
Since $p_k(n) > 0$ for all $k \geq 1$ and $n \geq 0$, it is enough to show that
\begin{equation}
    \tilde{a}_{j,+}^{(w-2)} > 0 \quad \text{for all}\quad 1 \le j \le n_{d,+}
\end{equation}
to prove that $b_{0,+}>0$.

For a quasimodular form $G = \sum_{j=0}^{r} E_2^j G_{j}$, define $\delta G$ as
\begin{equation}
    \label{eqn:delta}
    \delta G = \frac{\partial G}{\partial E_2} := \sum_{j=1}^{r} j E_2^{j-1} G_j.
\end{equation}
If $G$ has weight $w$ and depth $r \geq 1$, then $\delta G$ has weight $w - 2$ and depth $r - 1$.
Also, $\delta E_2 = 1$, and $\delta G = 0$ if and only if $G$ is a modular form (i.e., has depth zero).
$\delta$ satisfies the Leibniz rule, i.e. $\delta(G_1 G_2) = (\delta G_1) G_2 + G_1 (\delta G_2)$ for any quasimodular forms $G_1, G_2$.
Also, by definition, we have $\delta F_w = B_{w-2} + 2 E_2 C_{w-4} = \widetilde{F}_{w-2}$.
The following lemma gives a relation between $\delta$ and the Serre derivative $\partial_k$.

\begin{lemma}
    \label{lem:delta_serre}
    Let $G$ be a quasimodular form of homogeneous weight $w$.
    Then we have
    \begin{equation}
        \label{eqn:delta_serre}
        \delta(\partial_k G) = \partial_k(\delta G) + \frac{w-k}{12} G.
    \end{equation}
\end{lemma}
\begin{proof}
    It is known that the three operators $D = q \frac{\dd}{\dd q}$, $\partial = -12\delta$, and $H$, acting by the weight, form an $\mathfrak{sl}_2$-triple \cite[p. 468]{kaneko2006extremal}.
    In particular, we have $[D, \partial] = H$, or equivalently,
    \begin{equation}
        \label{eqn:D_delta}
        D(\delta G) - \delta(DG) = -\frac{w}{12} G
    \end{equation}
    for a quasimodular form $G$ of weight $w$.
    By $D = \partial_k + \frac{k}{12} E_2$, we have
    \begin{align*}
        D(\delta G) &= \partial_k (\delta G) + \frac{k}{12} E_2(\delta G) \\
        \delta(DG) &= \delta(\partial_k G) + \frac{k}{12} \delta(E_2 G) = \delta(\partial_k G) + \frac{k}{12} E_2(\delta G) + \frac{k}{12} G
    \end{align*}
    and plugging these into \eqref{eqn:D_delta} gives \eqref{eqn:delta_serre}.
\end{proof}

\begin{lemma}
    \label{lem:Ftilde_rec}
    For $w \geq 8$ with $w \equiv 0 \pmod{4}$, we have
    \begin{equation}
        \label{eqn:Ftilde_rec}
        \widetilde{F}_{w+2} = \frac{3(w-4)w}{16(w-10)(w-5)(w-3)(w+2)} \left(\mathcal{S}_w \widetilde{F}_{w-2} - \frac{1}{3} \partial_{w-1} F_{w}\right)
    \end{equation}
    where
    \begin{equation}
        \label{eqn:S_w}
        \mathcal{S}_w = \frac{(w-6)(w-5)}{36} E_4 - \partial_w \partial_{w-2} = -L_{2,w-2}^{-(w-10)(w-4)/48}
    \end{equation}
\end{lemma}
\begin{proof}
    By taking $\delta$ on both sides of \eqref{eqn:poseig4}$_{w}$, we have
    \begin{align*}
    \widetilde{F}_{w+2} &= \frac{3(w-4)w}{16(w-10)(w-5)(w-3)(w+2)} \left(\frac{(w-6)(w-5)}{36} \delta(E_4 F_w) - \delta (\partial_w \partial_{w-2} F_w)\right) \\
    &= \frac{3(w-4)w}{16(w-10)(w-5)(w-3)(w+2)} \left(\frac{(w-6)(w-5)}{36} E_4 (\delta F_w) - \partial_w(\delta (\partial_{w-2} F_w)) - \frac{1}{6} \partial_{w-2} F_w\right) \\
    &= \frac{3(w-4)w}{16(w-10)(w-5)(w-3)(w+2)} \left(\frac{(w-6)(w-5)}{36} E_4 (\delta F_w) - \partial_w(\partial_{w-2} (\delta F_w)) - \frac{1}{6} \partial_w F_w - \frac{1}{6} \partial_{w-2} F_w\right) \\
    &= \frac{3(w-4)w}{16(w-10)(w-5)(w-3)(w+2)} \left(\frac{(w-6)(w-5)}{36} E_4 \widetilde{F}_{w-2} - \partial_w(\partial_{w-2} \widetilde{F}_{w-2}) - \frac{1}{3} \partial_{w-1} F_w\right) \\
    &= \frac{3(w-4)w}{16(w-10)(w-5)(w-3)(w+2)} \left(\mathcal{S}_w \widetilde{F}_{w-2} - \frac{1}{3} \partial_{w-1} F_{w}\right).
    \end{align*}
\end{proof}

\begin{proposition}
    \label{prop:Ftilde_pos}
    For $w \equiv 0 \pmod{4}$ and $w \ge 12$, we have 
    \begin{equation}
        \tilde{a}_{j,+}^{(w-2)} > 0 \quad \text{for all} \quad 1 \le j \le \frac{w}{4} - 3, \quad \tilde{a}_{\frac{w}{4}-2,+}^{(w-2)} \ge \frac{1}{360}.
    \end{equation}
\end{proposition}
\begin{proof}
    Let $\alpha_w = - \frac{(w-10)(w-4)}{48}$.
    From $\cS_w = -L_{2,w-2}^{\alpha_w}$, Lemma \ref{lem:KZ2_coeff} implies that the $n$-th Fourier coefficient of $\cS_w \wtilde{F}_{w-2}$ is
    \[
        [q^n](\cS_w \wtilde{F}_{w-2}) = -\kappa_{2,w-2}^{\alpha_w}(n) \tilde{a}_{n,+}^{(w-2)} - \sum_{j=1}^{n-1} K_{2,w-2}^{\alpha_w}(n, j) \tilde{a}_{j,+}^{(w-2)}
    \]
    where
    \begin{align}
        \kappa_{2,w-2}^{\alpha_w}(n) &= \frac{(12n + w - 10)(4n - w + 4)}{48} \label{eqn:kappa_F}, \\
        K_{2,w-2}^{\alpha_w}(n,j) &= -2(w-1)\bigl(w(n-j)-2n\bigr)\sigma_1(n-j) - 5(w-10)(w-4)\sigma_3(n-j). \label{eqn:K_F}
    \end{align}
    We now use induction on $w$.
    For $w = 12$, direct computation shows that
    \[
    \wtilde{F}_{10} = \frac{1}{360} E_4 X_{6,1}
    \]
    which is completely positive, i.e. $\tilde{a}_{j,+}^{(10)} > 0$ for all $j \ge 1$.
    Also, $\tilde{a}_{1,+}^{(10)} = \frac{1}{360}$.
    Assume that $\tilde{a}_{j,+}^{(w-2)} > 0$ for all $1 \le j \le \frac{w}{4} - 3$ and $\tilde{a}_{\frac{w}{4}-2,+}^{(w-2)} \ge \frac{1}{360}$.
    Then, by \eqref{eqn:Ftilde_rec}, \eqref{eqn:kappa_F}, and \eqref{eqn:K_F}, the $n$-th Fourier coefficient of $\wtilde{F}_{w+2}$ is
    \begin{equation}
        \label{eqn:Ftilde_induction_coeff}
        c_w\left(-\kappa_{2,w-2}^{\alpha_w}(n) \tilde{a}_{n,+}^{(w-2)} - \sum_{j=1}^{n-1} K_{2,w-2}^{\alpha_w}(n,j)\tilde{a}_{j,+}^{(w-2)} - \frac{1}{3}n b_n^{(w)} + \frac{w-1}{36}\left(b_n^{(w)} - 24 \sum_{k=1}^{n-1} \sigma_1(n-k) b_k^{(w)}\right)\right)
    \end{equation}
    where $c_w = \frac{3(w-4)w}{16(w-10)(w-5)(w-3)(w+2)}>0$.
    Since $F_w = q^{\frac{w}{4} - 1} + O(q^{\frac{w}{4}})$, we have $b_n^{(w)} = 0$ for all $1 \leq n \leq \frac{w}{4} - 2$.
    In this range, $4n-w+4<0$, so \eqref{eqn:kappa_F} gives $\kappa_{2,w-2}^{\alpha_w}(n)<0$. Moreover, for $1\leq j\leq n-1$ and $n\leq\frac{w}{4}-1$,
    \[
        w(n-j)-2n \geq w-2n>0,
    \]
    and hence \eqref{eqn:K_F} gives $K_{2,w-2}^{\alpha_w}(n,j)<0$.
    For such $n$, \eqref{eqn:Ftilde_induction_coeff} becomes
    \[
        c_w\left(- \kappa_{2,w-2}^{\alpha_w}(n) \tilde{a}_{n,+}^{(w-2)} - \sum_{j=1}^{n-1} K_{2,w-2}^{\alpha_w}(n,j)\tilde{a}_{j,+}^{(w-2)}\right)
    \]
    which is positive by the induction hypothesis.
    For $n = \frac{w}{4} - 1$, normalization gives $b_n^{(w)}=1$ and $b_k^{(w)}=0$ for $k<n$, while $\kappa_{2,w-2}^{\alpha_w}(n) = 0$ and
    \[
    K_{2,w-2}^{\alpha_w}\left(\frac{w}{4}-1, \frac{w}{4}-2\right) = -2(w-1)\left(w - 2\left(\frac{w}{4}-1\right)\right) - 5(w-10)(w-4) = -(6w^2 - 67w + 196),
    \]
    hence \eqref{eqn:Ftilde_induction_coeff} can be bounded from below as
    \begin{align}
        c_w\left(-\sum_{j=1}^{n-1} K_{2,w-2}^{\alpha_w}(n,j)\tilde{a}_{j,+}^{(w-2)} - \frac{2w-11}{36}\right) &\ge c_w \left(-K_{2,w-2}^{\alpha_w}\left(\frac{w}{4}-1, \frac{w}{4}-2\right) \tilde{a}_{\frac{w}{4}-2,+}^{(w-2)} - \frac{2w-11}{36}\right) \nonumber \\
        &\ge c_w \left(\frac{6w^2 - 67w + 196}{360} - \frac{2w - 11}{36}\right) \nonumber \\
        &= \frac{(w-6)(w-4)w(2w-17)}{640(w-10)(w-5)(w-3)(w+2)} \label{eqn:boundary_coeff}
    \end{align}
    and \eqref{eqn:boundary_coeff} is greater than or equal to $\frac{1}{360}$ if and only if $2w^4 - 77w^3 + 1018 w^2 - 4312 w + 4800 \ge 0$, which is true for all $w \ge 12$.
    This proves the claim for $n = \frac{w}{4} - 1$ and completes the induction.
\end{proof}

Since $n_{d,+} \le \frac{w}{4} - 2$, Proposition \ref{prop:Ftilde_pos} implies that $b_{0,+} > 0$, and hence $M_{d,+}(\mathbf{0}) < 0$.
As a corollary, we obtain the desired upper bound of $\rA_+(d)$.

\begin{corollary}
    \label{cor:Apos_0mod8}
    For $d \equiv 0 \pmod{8}$, we have
    \[
        \rA_+(d) < \sqrt{2n_{d,+}} = \sqrt{2\left(\left\lfloor\frac{d}{16}\right\rfloor + 1\right)}.
    \]
\end{corollary}
\begin{proof}
    By Theorem \ref{thm:poseigpos} and \eqref{eqn:Mdpos}, $M_{d,+}(\bx) \geq 0$ for all $\|\bx\| > \sqrt{2n_{d,+}}$.
    Also, by Proposition \ref{prop:Ftilde_pos}, \eqref{eqn:Mpos_orig}, and \eqref{eqn:bzero_pos}, we have $M_{d,+}(\mathbf{0}) < 0$.
    Hence, if we consider the function $f(\bx) := M_{d,+}(\bx) - M_{d,+}(\mathbf{0}) e^{-\pi \|\bx\|^2}$, then $\what{f} = f$, $f(\mathbf{0}) = 0$, and there exists $c > 0$ such that $f(\bx) > 0$ for all $\|\bx\| > \sqrt{2n_{d,+}} - c$.
    Hence $f \in \mathcal{A}_+(d)$ with $r(f) = r(\what{f}) < \sqrt{2n_{d,+}}$, which implies the inequality.
\end{proof}

\begin{remark}
    Based on experimental evidence, we conjecture that $\widetilde{F}_{w-2}$ is completely positive for all $w \geq 8$ with $w \equiv 0 \pmod{4}$.
\end{remark}

\subsection{$(-1)^{d/4 + 1}$-eigenforms}
\label{subsec:fgh_negeig}

The construction of $(-1)^{d/4 + 1}$-eigenfunctions in \cite{feigenbaum2021eigenfunctions} differs from that of the $(-1)^{d/4}$-eigenfunctions above.
The corresponding ``modular forms'' $\{\phi_w\}_{w \geq 8}$ can be expressed in terms of the Jacobi theta functions $\Theta_2$ and $\Theta_4$ in \eqref{eqn:jacobi_theta}, the modular discriminant $\Delta$, and $\log\lambda$.
It is more convenient for us to work with the $S$-transforms $\phi_w|_S$, and the forms $G_w$ below are the normalizations of these (so that the first nonzero Fourier coefficient of each $G_w$ is $1$); accordingly, $\cL = \log \lambda$ is replaced by $\cL_S = \log \lambda_S$.
The derivative of $\cL_S$ is \cite[Appendix]{feigenbaum2021eigenfunctions}
\begin{equation}
    \label{eqn:LS_derivative}
    \cL_S' = \frac{H_4'}{H_4} - \frac{H_2' + H_4'}{H_2 + H_4} = -\frac{1}{2} H_2.
\end{equation}
By the $\mathrm{SL}_2(\mathbb{Z})$-equivariance of the Serre derivative, the transforms $\phi_w|_S$ satisfy the same recurrence relations and differential equations as $\phi_w$.
We also find two missing pieces, $G_4$ and $G_6$: after adjoining them, the low-weight instances $w=4$ of \eqref{eqn:negeigrec4} and $w=8$ of \eqref{eqn:negeigrec2} hold, and $G_4$ satisfies \eqref{eqn:negeigde}.
The theorem below is essentially equivalent to the combination of Theorem 4.4 and Propositions 5.5 and 5.6 of \cite{feigenbaum2021eigenfunctions}, with minor modifications.\footnote{There are minor errors in the expressions for $\phi_{16}$ and $\phi_{18}$ in \cite{feigenbaum2021eigenfunctions}; they are corrected in Theorem \ref{thm:fghnegeig}.}

\begin{theorem}[Feigenbaum--Grabner--Hardin \cite{feigenbaum2021eigenfunctions}, normalized]
\label{thm:fghnegeig} 
For even $w \geq 4$, define ``modular forms'' $\{G_w\}_{w \geq 4}$ of weight $w$ and level $\Gamma(2)$ by
\begin{align*}
    G_{4} &= \frac{H_2}{2^5} (H_2 + 2H_4) \\
    G_{6} &= \frac{H_2}{2^4} (H_2^2 + H_2 H_4 + H_4^2) \\
    G_{8} &= \frac{H_{2}^{3}}{2^{13}}(H_{2} + 2H_{4})\\
    G_{10} &= \frac{H_{2}^{3}}{2^{12} \cdot 5}(2H_{2}^{2} + 5 H_{2} H_{4}+ 5H_{4}^{2}) \\
    G_{12} &= \frac{3 \Delta \cL_S}{2^{11} \cdot 7} + \frac{3H_{2}^{3}}{2^{20} \cdot 7}(H_{2}^{3} + 3 H_{2}^{2}H_{4} + 3H_{2} H_{4}^{2} + 2H_{4}^{3})\\
    G_{14} &= \frac{H_{2}^{5}}{2^{20} \cdot 7}(2H_{2}^{2} + 7H_{2} H_{4} + 7 H_{4}^{2})\\
    G_{16} &= \frac{5E_4 \Delta \cL_S}{2^{18} \cdot 11} + \frac{5H_{2}^{3}}{2^{29} \cdot 3 \cdot 11} (5 H_{2}^{5} + 20 H_{2}^{4} H_{4} + 42 H_{2}^{3}H_{4}^{2} + 68 H_{2}^{2} H_{4}^{3} + 60 H_{2} H_{4}^{4} + 24 H_{4}^{5}) \\
    G_{18} &= -\frac{5E_6 \Delta \cL_S}{2^{17} \cdot 11 \cdot 13}  +\frac{5H_{2}^{3}}{2^{27} \cdot 3 \cdot 11 \cdot 13} (10 H_{2}^{6} + 45 H_{2}^{5} H_{4} + 68  H_{2}^{4} H_{4}^{2} + 34 H_{2}^{3}H_{4}^{3} -13H_{2}^{2} H_{4}^{4}-36 H_{2}H_{4}^{5} -12 H_{4}^{6})
\end{align*}
The recurrence \eqref{eqn:negeigrec2} holds for $w=8,12$ and for $w \equiv 0 \pmod{4}$ with $w \geq 20$, whereas \eqref{eqn:negeigrec4} holds for $w=4$ and for $w \equiv 0 \pmod{4}$ with $w \geq 12$:
\begin{align}
    G_{w + 2} &= \frac{3(w-6)(w-2)}{16(w-16)(w-5)(w-4)(w-3)} \left(\frac{(w-9)(w-8)}{36}E_4 G_{w-2} - \partial_{w-2}^{2} G_{w-2} \right) \label{eqn:negeigrec2} \\
    G_{w + 4} &= \frac{3(w-2)(w+2)}{16(w-8)(w-3)(w-1)(w+4)} \left(\frac{(w-4)(w-3)}{36} E_4 G_{w} - \partial_{w}^{2} G_{w}\right). \label{eqn:negeigrec4}
\end{align}
Then the vanishing order of $G_{w}$ at the cusp is $\lfloor \frac{w}{4} \rfloor - \frac{1}{2}$.
For $w \equiv 0 \pmod{4}$, the functions $G_{w}$ satisfy the third-order ordinary differential equation
\begin{equation}
\label{eqn:negeigde}
L_{3,w}^{(\frac{w-2}{4},0)} G_w = \partial_{w}^3G_w - \frac{3w^2 - 24w + 80}{144}E_4 \partial_{w}G_{w} - \frac{(w-12)(w-3)w}{864}E_6 G_w = 0,
\end{equation}
or equivalently,
\begin{equation}
    \label{eqn:negeigde2}
    G_w''' - \frac{w+2}{4} E_2 G_w'' + \left(\frac{(w+1)(w+2)}{4}E_2' + \frac{w-2}{4}E_4\right) G_w' - \left(\frac{w(w+1)(w+2)}{24}E_2'' + \frac{(w-2)w}{16} E_4'\right) G_w = 0.
\end{equation}
Now, let $d$ be a positive integer divisible by $4$ and $n_{d,-} = \lfloor d/16\rfloor + 1$.
Let $w = w_{d, -} = 12 \lfloor d / 16 \rfloor - d/2 + 14$.
Then the following function
\begin{equation}
\label{eqn:Mdneg}
    M_{d,-}(\mathbf{x}) = 4 \sin^{2}\left(\frac{\pi \|\mathbf{x}\|^{2}}{2}\right) \int_{0}^{\infty} \frac{t^{-w}G_{w}(i/t)}{\Delta(it)^{n_{d,-}}} e^{-\pi \|\mathbf{x}\|^{2} t} \dd t
\end{equation}
for $\mathbf{x}\in \mathbb{R}^{d}$ satisfies (here we abuse notation by writing $M_{d,-}(\mathbf{x}) = M_{d,-}(\|\mathbf{x}\|)$)
\begin{align*}
    \widehat{M_{d,-}}(\mathbf{x}) &= (-1)^{d/4 + 1} M_{d,-}(\mathbf{x}) \quad \forall \mathbf{x} \in \mathbb{R}^{d}, \\
    M_{d,-}(\sqrt{2 n_{d,-}}) &= 0\quad\text{and}\quad M_{d,-}'(\sqrt{2n_{d,-}}) \neq 0, \\
    M_{d,-}(\sqrt{2m}) &= M_{d,-}'(\sqrt{2m}) = 0 \quad\forall m > n_{d,-}, m \in \mathbb{Z}.
\end{align*}
Note that the integral \eqref{eqn:Mdneg} converges when $\|\bx\| > \sqrt{2n_{d,-}}$, and one can analytically continue it to the origin (see Section \ref{subsec:nonposorig_neg}).
\end{theorem}
\begin{proof}
    The result is proved in \cite{feigenbaum2021eigenfunctions}, except for the normalization claim, which follows from \eqref{eqn:negeigde2} and induction on $w$, as in the proof of Theorem \ref{thm:fghposeig}.
\end{proof}

The functions $G_w$ are not modular forms in general; rather, they are combinations of modular forms of levels $1$ and $2$ and terms involving $\cL_S$.
$G_w$ admits a Fourier expansion in $q^{1/2} = e^{\pi i z}$ and is normalized in the sense that its first nonzero Fourier coefficient is $1$.

These forms also satisfy a relation similar to that in Proposition \ref{prop:serrederposeig}, which can likewise be proved by induction.
We omit the details of the proof.
\begin{proposition}
\label{prop:serredernegeig}
For $w \geq 8$ divisible by $4$, we have
\begin{align}
    G_{w+2} &= \frac{(w-6)(w-2)}{768(w-5)(w-3)} (E_4 G_{w-2} -E_6 G_{w-4}) \label{eqn:negeigeq1} \\
    \partial_{w}G_{w} &= \frac{w-3}{6} G_{w + 2} \label{eqn:negeigeq2}
\end{align}
\end{proposition}

We will assume $d \equiv 4 \pmod{8}$, so that
\begin{equation}
    \label{eqn:wdneg}
    w_{d,-} = \begin{cases}
        \frac{d}{4} + 11 & d \equiv 4 \pmod{16} \\ \frac{d}{4} + 5 & d \equiv 12 \pmod{16}
    \end{cases},\quad
    n_{d,-} = \begin{cases}
        \frac{d + 12}{16} = \frac{w}{4} - 2 & d \equiv 4 \pmod{16} \\ \frac{d + 4}{16} = \frac{w}{4} - 1 & d \equiv 12 \pmod{16}
    \end{cases}.
\end{equation}

\subsubsection{Companion of $X_{w,2}$}
\label{subsec:Yw}

We ask whether there is a family of level $\Gamma(2)$ ``extremal forms'' closely related to $G_{w}$, analogous to the relations between $F_{w}$ and $X_{w, 2}$ in Proposition \ref{prop:ext2poseig}.
The coefficients in \eqref{eqn:negeigeq1} and \eqref{eqn:negeigeq2} are obtained from those in \eqref{eqn:poseigeq1} and \eqref{eqn:poseigeq2} by shifting $w$ by $2$; this observation leads to the following family $\{Y_{w}\}_{w \ge 2}$.

\begin{definition}
Define $Y_{w}$ for even $w \geq 2$ inductively by
\begin{align*}
    Y_{2} &= \frac{H_{2}}{2^4} \\
    Y_{4} &= -\frac{3 E_4 \cL_S}{2^{11} \cdot 5} -\frac{3}{2^{12} \cdot 5} (H_{2}^{2} + 2 H_{2} H_{4}) \\
    Y_{6} &= 0 \\
    Y_{8} &= -\frac{5 E_4^2 \cL_S}{2^{19} \cdot 7} - \frac{5}{2^{21} \cdot 3 \cdot 7} (11 H_2^4 + 28 H_2^3 H_4 + 18 H_2^2 H_4^2 + 12 H_2 H_4^3),
\end{align*}
and for $w \geq 8$ with $w \equiv 0 \pmod{4}$,
\begin{align}
    Y_{w+4} &=  \frac{3(w+6)^{2}}{16 (w+3)(w+4)^2 (w+5)}\left(\frac{(w+2)(w+3)}{36} E_4 Y_{w} - \partial_{w}^{2} Y_{w}\right) \label{eqn:Ywdef1}\\
    Y_{w+2} &= \frac{6}{w + 3} \partial_{w} Y_{w} \label{eqn:Ywdef2}
\end{align}
\end{definition}

The $q$-expansions of $Y_w$ for $w \le 10$ are as follows:
\begin{align*}
    Y_{2} &= \frac{H_{2}}{2^4} = q^{\frac{1}{2}} + 4 q^{\frac{3}{2}} + 6 q^\frac{5}{2} + 8 q^\frac{7}{2} + 13 q^\frac{9}{2} + \cdots \\
    Y_{4} &= -\frac{3 E_4 \cL_S}{2^{11} \cdot 5} -\frac{3}{2^{12} \cdot 5} (H_{2}^{2} + 2 H_{2} H_{4}) \\
    &= q^{\frac{3}{2}} + \frac{276}{25}q^\frac{5}{2} + \frac{1566}{35}q^{\frac{7}{2}} + \frac{14072}{105}q^\frac{9}{2} + \frac{113963}{385}q^\frac{11}{2}+ \cdots \\
    Y_{6} &= 0 \\
    Y_{8} &= -\frac{5 E_4^2 \cL_S}{2^{19} \cdot 7} - \frac{5}{2^{21} \cdot 3 \cdot 7} (11 H_2^4 + 28 H_2^3 H_4 + 18 H_2^2 H_4^2 + 12 H_2 H_4^3) \\
    &= q^\frac{5}{2} + \frac{1020}{49}q^\frac{7}{2} + \frac{80470}{441}q^\frac{9}{2} + \frac{1593080}{1617}q^\frac{11}{2} + \frac{27913055}{7007}q^{\frac{13}{2}} + \cdots \\
    Y_{10} &= \frac{5 E_4 E_6 \cL_S}{2^{17} \cdot 7 \cdot 11} - \frac{5}{2^{19} \cdot 3 \cdot 7 \cdot 11} (2 H_2^5 + 11 H_2^4 H_4 - H_2^3 H_4^2 - 24 H_2^2 H_4^3 - 12 H_2 H_4^4) \\
    &= q^\frac{5}{2} + \frac{2004}{49}q^\frac{7}{2} + \frac{259918}{441}q^\frac{9}{2} + \frac{84839768}{17787}q^\frac{11}{2} + \frac{26865297}{1001}q^\frac{13}{2} + \cdots.
\end{align*}

Like $G_{w}$, these are \emph{not} modular forms in general because of the term $\cL_S$.
However, $Y_w$ satisfy certain recurrence relations and differential equations analogous to those of $X_{w, 2}$.
\begin{proposition}
\label{prop:Yproperties}
\begin{enumerate}
    \item For $w \geq 4$ with $w \equiv 0 \pmod{4}$, $Y_{w}$ satisfies the third-order ordinary differential equation
    \begin{equation}
        \label{eqn:Ywode}
        L_{3,w} Y_{w} = \partial_{w}^{3}Y_{w} - \frac{3(w+2)^2 - 4}{144} E_4 \partial_{w} Y_{w} - \frac{w^2(w+3)}{864} E_6 Y_{w} = 0
    \end{equation}
    where $L_{3,w} = L_{3,w}^{(0,0)}$ is the third-order Kaneko--Zagier operator \eqref{eqn:KZ3_def_serre}, or equivalently,
    \begin{equation}
        \label{eqn:Ywode2}
        Y_{w}''' - \frac{w+2}{4} E_2 Y_{w}'' + \frac{(w+1)(w+2)}{4} E_2' Y_{w}' - \frac{w(w+1)(w+2)}{24} E_2'' Y_w = 0.
    \end{equation}
    \item For $w \geq 8$ with $w \equiv 0 \pmod{4}$, we have
    \begin{align}
        Y_w &= q^{\frac{w + 2}{4}} + \frac{2(w + 2)(w^2 + 16w + 12)}{(w + 6)^2} q^{\frac{w + 6}{4}} + O(q^{\frac{w + 10}{4}}) \label{eqn:Yw4_qexp} \\
        Y_{w+2} &= q^{\frac{w + 2}{4}} + \frac{2(w^3 + 30 w^2 + 188w + 72)}{(w + 6)^2} q^{\frac{w + 6}{4}} + O(q^{\frac{w + 10}{4}}) \label{eqn:Yw4p2_qexp}
    \end{align}
    Together with the initial expansions above, this shows that the order of $Y_{w}$ at the cusp is $\lfloor \frac{w}{4} \rfloor + \frac{1}{2}$ for every even $w \geq 2$ with $w \neq 6$.
    \item For $w \geq 12$ with $w \equiv 0 \pmod{4}$, we have
    \begin{align}
        Y_{w+2} &= \frac{3(w+2)^{2}}{16(w-4)^{2}(w+1)(w+3)} \left(\frac{(w-3)(w-2)}{36} E_4 Y_{w-2} - \partial_{w-2}^{2} Y_{w-2}\right) \label{eqn:Yweq1}\\
        Y_{w+2} &= \frac{(w+2)^2}{768(w+1)(w+3)} (E_{4} Y_{w-2} - E_{6} Y_{w-4}) \label{eqn:Yweq2} \\
        G_{w} &= -\frac{256(w-1)w(w+1)}{(w-2)(w+2)^{2}} Y_{w} + E_{4} Y_{w-4} \label{eqn:Yweq3} \\
        G_{w+2} &= \frac{2(w-1)}{3(w-2)} E_{4} Y_{w-2} + \frac{w-4}{3(w-2)} E_{6} Y_{w-4}. \label{eqn:Yweq4}
    \end{align}
    The identity \eqref{eqn:Yweq3} also holds for $w=8$; the other three identities do not.
    \item For every even $w \ge 2$, $Y_w$ can be expressed as
    \[
        Y_w = \widetilde{Y}_{w} \cL_S + \Phi_w
    \]
    where $\widetilde{Y}_w$ and $\Phi_w$ are holomorphic modular forms of weight $w$ and levels $1$ and $2$, respectively.
\end{enumerate}
\end{proposition}
\begin{proof}
We prove (1) using Lemma~\ref{lem:intertwine}.
Set $k = w$ and choose
\[
    (\alpha,\beta,\gamma) = \left(0,0,-\frac{(w+2)(w+4)}{48}\right), \quad
    (\alpha',\beta',\gamma') = \left(0,0,-\frac{(w+4)(w+6)}{48}\right).
\]
Then we have
\[
    L_{3,w+4} L_{2,w}^{-\frac{(w+2)(w+4)}{48}} = L_{2,w+6}^{-\frac{(w+4)(w+6)}{48}} L_{3,w}.
\]
By \eqref{eqn:Ywdef1}, $Y_{w+4}$ is a nonzero constant multiple of $L_{2,w}^{-\frac{(w+2)(w+4)}{48}} Y_w$.
The intertwining relation therefore proves \eqref{eqn:Ywode} by induction on $w$; the base cases $L_{3,4}Y_4 = 0$ and $L_{3,8}Y_8 = 0$ follow by direct computation.
Statements (2)--(4) follow by induction from \eqref{eqn:Ywdef1}, \eqref{eqn:Ywdef2}, and \eqref{eqn:Ywode}, together with the initial values and recurrence relations for $G_w$ in Theorem \ref{thm:fghnegeig} when needed; we omit the routine details.
\end{proof}

Note that the ``coefficients'' in \eqref{eqn:Ywdef1}, \eqref{eqn:Ywdef2}, \eqref{eqn:Ywode}, \eqref{eqn:Ywode2}, \eqref{eqn:Yweq1}, \eqref{eqn:Yweq2}, \eqref{eqn:Yweq3}, and \eqref{eqn:Yweq4} coincide with those in \eqref{eqn:d2eq1}, \eqref{eqn:d2eq2}, \eqref{eqn:d2ode3}, \eqref{eqn:d2ode2}, \eqref{eqn:d2eq3}, \eqref{eqn:d2eq4}, \eqref{eqn:poseigeq3}, and \eqref{eqn:poseigeq4}, respectively, after replacing $w$ with $w + 2$.
Since $Y_w$ satisfies the Kaneko--Zagier differential equation \eqref{eqn:Ywode2} for $w \equiv 0 \pmod 4$, Nakaya's computation applies to $Y_w$ as well, so that for such $w$ the function $h := E_4^{-w/4} Y_w$, in terms of the variable $x = 1728/j(z)$, satisfies the hypergeometric differential equation
\begin{equation}
    \label{eqn:hpergeom}
    \begin{aligned}
    &x^2(1 - x) h'''(x) + x \left(-\frac{w-10}{4} + \frac{w-16}{4} x\right) h''(x) \\
    &\quad + \left(-\frac{w-2}{4} - \frac{3w^2 - 60w + 320}{144}x\right) h'(x) + \frac{(w-8)(w-4)w}{1728} h(x) = 0.
    \end{aligned}
\end{equation}
From this, we obtain the following hypergeometric expression of $Y_w$:
\begin{proposition}
\label{prop:Yhypergeom}
$Y_w$ admits the following hypergeometric series expansions, where the first formula holds for $k \geq 1$ and the second for $k \geq 0$ with $k \neq 1$:
\begin{align}
    Y_{4k}(z) &= j(z)^{-k - \frac{1}{2}} E_4(z)^{k} \cdot {}_{3}F_{2} \left(\frac{4k + 3}{6}, \frac{4k+5}{6}, \frac{4k+7}{6}; k + \frac{3}{2}, k + \frac{3}{2}; \frac{1728}{j(z)}\right), \\
    Y_{4k + 2}(z) &= j(z)^{-k -\frac{1}{2}} E_4(z)^{k - 1} E_6(z) \cdot {}_{3}F_{2} \left(\frac{4k + 5}{6}, \frac{4k+7}{6}, \frac{4k+9}{6}; k + \frac{3}{2}, k + \frac{3}{2}; \frac{1728}{j(z)}\right).
\end{align}
The first identity holds for $z=it$ and $t\geq1$; the second holds for $t>1$ and extends continuously to a strictly positive value at $t=1$, although the hypergeometric series itself diverges there.
\end{proposition}
\begin{proof}
The proof is similar to that of Theorem \ref{thm:extd2nakaya}.
\end{proof}

Now we prove that $Y_w$ is positive for all even $w \ge 2$ with $w \ne 6$.
We need some auxiliary elementary inequalities, which can be proved by calculus.
\begin{lemma}
    \label{lem:logineq1}    
    For $x > 0$, we have
    \begin{equation}
        \label{eqn:logineq1}
        \log(1 + x) > \frac{2x}{x+2}.
    \end{equation}
\end{lemma}
\begin{proof}
    Both sides vanish at $x = 0$, so the claim follows from
    \[
        \frac{\dd}{\dd x} \left(\log(1+x) - \frac{2x}{x+2}\right) = \frac{x^2}{(x+1)(x+2)^2} > 0.
    \]
\end{proof}
\begin{theorem}
    \label{thm:Ywpos}
    For every even $w \ge 2$ with $w \ne 6$, we have $Y_w(it)>0$ for all $t>0$.
\end{theorem}
\begin{proof}
    The proof is similar to that of Theorem \ref{thm:kkd2weak}, where we use Proposition \ref{prop:Yhypergeom} instead of Theorem \ref{thm:extd2nakaya} and the recurrence relations \eqref{eqn:Yweq2} and \eqref{eqn:Ywdef2} instead of \eqref{eqn:d2eq4} and \eqref{eqn:d2eq2}.
    We first prove the low-weight cases $Y_2$, $Y_4$, $Y_8$, and $Y_{10}$; the last two will serve as the base cases for induction.
    Positivity of $Y_2$ is clear from its expression, since $H_2(it) > 0$ for all $t > 0$.
    For the other three cases, we first consider $Y_4$.
    By \eqref{eqn:e4theta}, we have
    \[
        Y_4(it) > 0 \Leftrightarrow - \cL_S(it) > \frac{H_2(it)(H_2(it) + 2H_4(it))}{2(H_2(it)^2 + H_2(it) H_4(it) + H_4(it)^2)}.
    \]
    If we put $x = H_2(it) / H_4(it)$, then $x > 0$, and $-\cL_S(it) = -\log(1 - \lambda(it)) = \log(1+x)$, so the above inequality is equivalent to
    \begin{equation}
        \label{eqn:logineqY4}
        \log(1 + x) > \frac{x(x+2)}{2(x^2 + x + 1)}.
    \end{equation}
    This follows from Lemma \ref{lem:logineq1}; we have
    \begin{align*}
        \log(1 + x) - \frac{x(x+2)}{2(x^2 + x + 1)} > \frac{2x}{x+2} - \frac{x(x+2)}{2(x^2 + x + 1)} = \frac{3x^3}{2(x+2)(x^2 + x + 1)} > 0.
    \end{align*}
    Similarly, by \eqref{eqn:e4theta} and \eqref{eqn:e6theta}, positivity of $Y_8$ and $Y_{10}$ follows from the inequalities
    \begin{align}
        (x^2+x+1)^2 \log(1+x) &> \frac{1}{12} (11 x^4 + 28 x^3 + 18 x^2 + 12 x) \label{eqn:logineqY8} \\
        -(x^2+x+1)(2x+1)(x+2)(1-x) \log(1+x) &> \frac{1}{6} (2 x^5 + 11 x^4 - x^3 - 24 x^2 - 12 x). \label{eqn:logineqY10}
    \end{align}
    For the first inequality \eqref{eqn:logineqY8}, divide both sides by $(x^2+x+1)^2$; then both sides vanish at $x = 0$, and the claim follows from
    \[
        \frac{\dd}{\dd x} \left(\log(1 + x) - \frac{11x^4 + 28x^3 + 18x^2 + 12x}{12(x^2 + x + 1)^2}\right) = \frac{x^4 (2x^2 + 7x + 7)}{2(x+1)(x^2 + x + 1)^3} > 0.
    \]
    For the second inequality \eqref{eqn:logineqY10}, note that we only need to prove it for $x \ge 1$, which corresponds to $0 < t \le 1$, since the case $t > 1$ follows from Proposition \ref{prop:Yhypergeom}.
    Indeed, $\frac{\dd}{\dd t} \cL_S(it) = \pi H_2(it) > 0$ by \eqref{eqn:LS_derivative}, so $\lambda(it) = 1 - e^{\cL_S(it)}$ is strictly decreasing in $t$; by \eqref{eqn:lambda_def}, $x = \lambda(it) / (1 - \lambda(it))$ is then also strictly decreasing in $t$, and $x(1) = 1$ since $\lambda(i) = \frac{1}{2}$ by \eqref{eqn:lambdaS_def}.
    For $x \ge 1$ we have $1 - x \le 0$, so the coefficient $-(x^2+x+1)(2x+1)(x+2)(1-x)$ of $\log(1+x)$ is nonnegative, and Lemma \ref{lem:logineq1} gives
    \begin{align*}
        &-(x^2+x+1)(2x+1)(x+2)(1-x) \log(1+x) - \frac{1}{6} (2 x^5 + 11 x^4 - x^3 - 24 x^2 - 12 x) \\
        &\quad\ge -(x^2+x+1)(2x+1)(x+2)(1-x) \cdot \frac{2x}{x+2} - \frac{1}{6} (2 x^5 + 11 x^4 - x^3 - 24 x^2 - 12 x) \\
        &\quad= \frac{1}{6}x^3(22x^2 +x + 1) > 0.
    \end{align*}

    The induction step is similar to that in the proof of Theorem \ref{thm:kkd2weak}.
    One thing to note is that we can still apply Proposition \ref{prop:derpos} (4) even if $Y_{w}$ are not quasimodular forms in general.
    In fact, if $Y_{w+2}$ is positive, then $\partial_w Y_w$ is also positive by \eqref{eqn:Yweq2}, and hence
    \[
        \frac{\dd}{\dd t}\left(\frac{Y_w(it)}{\Delta(it)^{w/12}}\right)
        =-2\pi\frac{(\partial_w Y_w)(it)}{\Delta(it)^{w/12}}<0.
    \]
    Because $Y_w(i)>0$, this monotonicity implies $Y_w(it)>0$ for $0<t<1$.
    Thus the induction step establishes positivity of both $Y_w$ and $Y_{w+2}$; starting with $Y_8$ and $Y_{10}$ proves the claim in every higher weight.
\end{proof}

Now, positivity of $G_{w}$ follows from that of $Y_{w}$.

\begin{corollary}
\label{cor:negeigpos}
For all even $w \geq 8$, $G_{w}$ is positive.
\end{corollary}
\begin{proof}
    The proof is similar to that of Theorem \ref{thm:poseigpos}.
    We will use induction on $w \equiv 0 \pmod{4}$, where we assume that $G_{w-2}$ and $G_{w-4}$ are positive.
    By the same elementary argument as in Proposition \ref{prop:derpos} (4), together with \eqref{eqn:negeigeq2}, it is enough to show positivity for $G_{w+2}$ with $w \equiv 0 \pmod{4}$ and $w \geq 12$ (positivity of $G_w$ for $w = 4,6,8,10$ is clear from their expressions, since $H_2 = \Theta_2^4$ and $H_4 = \Theta_4^4$ are positive on the imaginary axis).
    As in the proof of Theorem \ref{thm:poseigpos}, we can still apply Proposition \ref{prop:derpos} (4) even if $G_w$ are not quasimodular forms in general.
\end{proof}

\begin{remark}
Based on experimental evidence, we conjecture that $Y_{w}$ and $G_{w}$ are completely positive; that is, all their Fourier coefficients are nonnegative.
This may be viewed as a corrected version of the conjecture proposed in \cite[Remark 6.4]{feigenbaum2021eigenfunctions}.
It may be possible to prove that all but finitely many coefficients are positive for each $G_{w}$ using Jenkins--Pratt's coefficient bounds for level 2 modular forms \cite{jenkins2014coefficient}.
\end{remark}

\subsubsection{Nonpositivity of $M_{d,-}(\mathbf{0})$}
\label{subsec:nonposorig_neg}

As in the case of $M_{d,+}(\mathbf{0})$, we can express $M_{d,-}(\mathbf{0})$ in terms of the Fourier coefficients of a modular form related to $G_{w}$.

\begin{proposition}
\label{prop:Gw_comp}
For each even $w \geq 4$, there exist a level 1 modular form $\widetilde{G}_{w-12}$ of weight $w-12$ and a level $\Gamma(2)$ modular form $\Psi_w$ of weight $w$ such that
\begin{equation}
    \label{eqn:Gw_decomp}
    G_w = \widetilde{G}_{w-12} \Delta \cL_S + \Psi_w.
\end{equation}
In particular, $\widetilde{G}_{w} = 0$ for $w = -8, -6, -4, -2, 2$ and $\widetilde{G}_{0} = \frac{3}{2^{11} \cdot 7}$, and for $w \equiv 0 \pmod{4}$ with $w \geq 0$, we have
\begin{equation}
    \label{eqn:Gtilde_rec}
    \widetilde{G}_{w+4} = \frac{3(w+10)(w+14)}{16(w+4)(w+9)(w+11)(w+16)} \left(\frac{(w+8)(w+9)}{36} E_4 \widetilde{G}_{w} - \partial_{w}^{2} \widetilde{G}_{w}\right)
\end{equation}
\end{proposition}
\begin{proof}
    The recurrence relations \eqref{eqn:negeigrec2} and \eqref{eqn:negeigrec4}, together with \eqref{eqn:LS_derivative}, show by induction on $w$ that $G_w$ can be written in the form \eqref{eqn:Gw_decomp} (note that $\partial_{12} \Delta = 0$).
    To prove \eqref{eqn:Gtilde_rec}, it is enough to show that the decomposition \eqref{eqn:Gw_decomp} is unique; equivalently, if $A$ is a level 1 modular form and $B$ is a level $\Gamma(2)$ modular form of the same weight, then $A \cL_S + B = 0$ only when $A=B=0$.
    If $A$ and $B$ are nonzero, then $\cL_S = -\frac{B}{A}$ must be a modular function for $\Gamma(2)$, and hence a rational function of the modular lambda function $\lambda$.
    In other words, there exists a rational function $R \in \bC(x)$ such that
    \[
    \cL_S(z) = \log (1 - \lambda(z)) = R(\lambda(z)),
    \]
    for all $z \in \bH$, which is impossible because $\log(1 - x)$ is not rational in $x$.
\end{proof}

The following proposition shows that $\wtilde{G}_w$ satisfies a third-order modular linear differential equation.
\begin{proposition}
    \label{prop:Gtilde_ode}
    For all $w \equiv 0 \pmod{4}$ and $w \geq 0$, $\widetilde{G}_{w}$ satisfies the third-order ordinary differential equation
    \begin{equation}
        \label{eqn:Gtilde_ode}
        L_{3,w}^{(-\frac{w+6}{4},0)} \wtilde{G}_w
        = \partial_w^3 \wtilde{G}_w - \frac{3w^2 + 48w + 224}{144} E_4 \partial_w \wtilde{G}_w - \frac{w(w+9)(w+12)}{864} E_6 \wtilde{G}_w = 0,
    \end{equation}
    or equivalently,
    \begin{equation}
        \label{eqn:Gtilde_ode2}
        \wtilde{G}_w''' - \frac{w+2}{4} E_2 \wtilde{G}_w'' + \left(\frac{(w+1)(w+2)}{4}E_2' - \frac{w+6}{4} E_4\right) \wtilde{G}_w' - \left(\frac{w(w+1)(w+2)}{24} E_2'' - \frac{w(w+6)}{16} E_4'\right) \wtilde{G}_w = 0.
    \end{equation}
\end{proposition}
\begin{proof}
    The following triples of parameters
    \[
        (\alpha,\beta,\gamma) = \left(-\frac{w+10}{4},0,-\frac{(w+6)(w+16)}{48}\right), \quad
        (\alpha',\beta',\gamma') = \left(-\frac{w+6}{4},0,-\frac{(w+10)(w+16)}{48}\right)
    \]
    satisfy the four conditions in Lemma~\ref{lem:intertwine}.
    Define $\cM_w := L_{3,w}^{(-\frac{w+6}{4},0)}$ and $\cN_w := L_{2,w}^{-\frac{(w+6)(w+16)}{48}} = \partial_w^2 - \frac{(w+8)(w+9)}{36} E_4$.
    By Lemma~\ref{lem:intertwine}, we have the following intertwining relation:
    \[
        \cM_{w+4} \cN_w = L_{2,w+6}^{-\frac{(w+10)(w+16)}{48}} \cM_w.
    \]
    By \eqref{eqn:Gtilde_rec}, $\wtilde{G}_{w+4}$ is a constant multiple of $\cN_w \wtilde{G}_w$.
    Hence \eqref{eqn:Gtilde_ode} follows from the intertwining relation and induction on $w$, where $\cM_0 \wtilde{G}_0 = 0$.
\end{proof}

We have $\cL_S(i/t) = \cL(it)$; using \eqref{eqn:L_qexp} and $w \equiv 0 \pmod 4$, we can write $G_{w}(i/t)$ as
\begin{align}
    t^{-w} G_w\left(\frac{i}{t}\right) &= t^{-w} \widetilde{G}_{w-12}\left(\frac{i}{t}\right) \Delta\left(\frac{i}{t}\right) \cL_S\left(\frac{i}{t}\right) + t^{-w} \Psi_w\left(\frac{i}{t}\right) \\
    &= \widetilde{G}_{w-12}(it) \Delta(it) \log \lambda(it) + (\Psi_w|_{w}S)(it) \\
    &= \widetilde{G}_{w-12}(it) \Delta(it) \left(-\pi t + 4 \log 2 + \sum_{k \ge 1}(-1)^k \frac{r_4(k)}{k} e^{-\pi k t}\right) + (\Psi_w|_{w}S)(it)
\end{align}
Thus, the integrand in \eqref{eqn:Mdneg} can be expressed as
\begin{align}
    \phi_{-}(t) &:= \frac{t^{-w} G_{w}(i/t)}{\Delta(it)^{n_{d,-}}} \\
    &= -\pi t \frac{\widetilde{G}_{w-12}(it)}{\Delta(it)^{n_{d,-} - 1}} + \frac{\widetilde{G}_{w-12}(it)\Delta(it)(\log \lambda(it) + \pi t) + (\Psi_w|_w S)(it)}{\Delta(it)^{n_{d,-}}}
\end{align}
where the numerator of the second term is $2$-periodic.
Write the Fourier expansions as
\begin{equation*}
    \frac{\widetilde{G}_{w-12}(z)}{\Delta(z)^{n_{d,-} - 1}} = \sum_{n \ge 1 - n_{d,-}} b_{n,-} q^n, \quad \frac{\widetilde{G}_{w-12}(z)\Delta(z)(\log \lambda(z) - \pi i z) + (\Psi_w|_w S)(z)}{\Delta(z)^{n_{d,-}}} = \sum_{n \ge - 2n_{d,-}} c_{n,-} q^{\frac{n}{2}}
\end{equation*}
and define the truncation $\widetilde{\phi}_{-}(t)$ of $\phi_{-}(t)$ as
\begin{equation}
    \widetilde{\phi}_{-}(t) := \phi_{-}(t) - \left(-\pi t \sum_{1-n_{d,-} \le k \le 0} b_{k,-} e^{- 2 \pi k t} + \sum_{-2n_{d,-} \le k \le 0} c_{k,-} e^{-\pi k t} \right).
\end{equation}
Then \eqref{eqn:Mdneg} can be rewritten as
\begin{align*}
    M_{d,-}(\mathbf{x}) &= 4 \sin^2\left(\frac{\pi\|\bx\|^2}{2}\right) \int_0^\infty \widetilde{\phi}_{-}(t) e^{-\pi \|\bx\|^2 t} \dd t \\
    &\quad + 4 \sin^2\left(\frac{\pi\|\bx\|^2}{2}\right) \left(-\pi \sum_{1-n_{d,-} \le k \le 0} b_{k,-} \int_0^\infty t e^{-\pi(\|\bx\|^2 + 2k) t} \dd t + \sum_{-2n_{d,-} \le k \le 0} c_{k,-} \int_0^\infty e^{-\pi(\|\bx\|^2 + k) t} \dd t\right) \\
    &= 4 \sin^2\left(\frac{\pi\|\bx\|^2}{2}\right) \int_0^\infty \widetilde{\phi}_{-}(t) e^{-\pi \|\bx\|^2 t} \dd t \\
    &\quad + 4 \sin^2\left(\frac{\pi \|\bx\|^2}{2}\right) \left[-\pi \sum_{1-n_{d,-} \le k \le 0} b_{k,-} \cdot \frac{1}{\pi^2(\|\bx\|^2 + 2k)^2} + \sum_{-2n_{d,-} \le k \le 0} c_{k,-} \cdot \frac{1}{\pi(\|\bx\|^2 + k)}\right]
\end{align*}
which analytically continues to $\mathbf{x} = \mathbf{0}$.
This gives
\begin{equation}
    \label{eqn:Mdnegzero}
    M_{d,-}(\mathbf{0}) = -\pi \cdot 4 \cdot \frac{\pi^2}{4} \cdot \frac{b_{0,-}}{\pi^2} = -\pi b_{0,-}.
\end{equation}
Thus, $M_{d,-}(\mathbf{0}) \le 0$ if and only if $b_{0,-} \ge 0$.
Write the Fourier expansion of $\widetilde{G}_{w-12}$ as
\begin{equation}
\widetilde{G}_{w-12} = \sum_{n \ge 0} \tilde{a}_{n,-}^{(w-12)} q^n.
\end{equation}
As in the case of $M_{d,+}(\mathbf{0})$, $b_{0, -}$ is given by
\begin{equation}
    \label{eqn:bzero_neg}
    b_{0,-} = [q^{n_{d,-} - 1}] \left(\sum_{n \ge 0} p_{24(n_{d,-} - 1)}(n) q^n\right) \left(\sum_{n \ge 0} \tilde{a}_{n,-}^{(w-12)} q^n\right) = \sum_{j=0}^{n_{d,-} - 1} p_{24(n_{d,-} - 1)}(n_{d,-} - 1 - j) \tilde{a}_{j,-}^{(w-12)}
\end{equation}
where $p_k(n)$ is defined as in \eqref{eqn:pkn}, hence $b_{0,-} \ge 0$ if $\tilde{a}_{n,-}^{(w-12)} \ge 0$ for all $0 \le n \le n_{d,-} - 1$.
Since $n_{d,-}$ is either $\frac{w}{4} - 2$ or $\frac{w}{4} - 1$, it is enough to show that $\tilde{a}_{n,-}^{(w)} \ge 0$ for all $0 \le n \le \frac{w}{4} + 1$.
We consider $0 \le n \le \frac{w}{4}$ and the boundary case $n = \frac{w}{4} + 1$ simultaneously, following the same general strategy as in the proof of Proposition \ref{prop:Ftilde_pos}.

\begin{proposition}
    \label{prop:Gtilde_pos}
    For all $w \equiv 0 \pmod{4}$ with $w \ge 0$, we have 
    \begin{equation}
        \tilde{a}_{n,-}^{(w)} \ge 0 \quad \text{for all } 0 \le n \le \frac{w}{4}.
    \end{equation}
\end{proposition}

The boundary case $n = \frac{w}{4} + 1$ is more delicate: in the induction step one must evaluate \eqref{eqn:Gtilde_coeff_rec} at $n = \frac{w}{4} + 2$, where $\kappa_{2,w}^{\alpha_w}(n) > 0$, so it is not immediately clear that the right-hand side is nonnegative.
In this case, we use the differential equation \eqref{eqn:Gtilde_ode2} to express $\tilde{a}_{\frac{w}{4}+2,-}^{(w)}$ as a linear combination of $\tilde{a}_{n,-}^{(w)}$ for $0 \le n \le \frac{w}{4}+1$.
\begin{proposition}
    \label{prop:Gtilde_pos_boundary}
    For all $w \equiv 0 \pmod{4}$ with $w \ge 0$, we have
    \begin{equation}
        \tilde{a}_{\frac{w}{4}+1,-}^{(w)} \ge 0.
    \end{equation}
\end{proposition}
We prove Propositions \ref{prop:Gtilde_pos} and \ref{prop:Gtilde_pos_boundary} together.
\begin{proof}
    We prove the two statements simultaneously by induction on $w$.
    Let $\alpha_w = -\frac{(w+6)(w+16)}{48}$.
    By \eqref{eqn:Gtilde_rec}, the $n$-th Fourier coefficient of $\wtilde{G}_{w+4}$ is
    \begin{equation}
        \label{eqn:Gtilde_coeff_rec}
        \tilde{a}_{n,-}^{(w+4)} = \frac{3(w+10)(w+14)}{16(w+4)(w+9)(w+11)(w+16)} \left(-\kappa_{2,w}^{\alpha_w}(n) \tilde{a}_{n,-}^{(w)} - \sum_{j=0}^{n-1} K_{2,w}^{\alpha_w}(n,j) \tilde{a}_{j,-}^{(w)}\right),
    \end{equation}
    where
    \begin{align}
        \kappa_{2,w}^{\alpha_w}(n) &= \frac{(12n + w + 16)(4n-w-6)}{48}, \label{eqn:kappa_G} \\
        K_{2,w}^{\alpha_w}(n,j) &= -2(w+1)(w(n-j) - 2j) \sigma_1(n-j) - 5(w+6)(w+16) \sigma_3(n-j). \label{eqn:K_G}
    \end{align}
    By applying Lemma \ref{lem:KZ3_coeff} to \eqref{eqn:Gtilde_ode}, we obtain a linear relation
    \begin{equation}
        \label{eqn:Gtilde_coeff_linear}
        \frac{n(n+1)(4n-w-6)}{4} \tilde{a}_{n,-}^{(w)} + \sum_{j=0}^{n-1} K_w'(n,j) \tilde{a}_{j,-}^{(w)} = 0
    \end{equation}
    where
    \begin{equation}
        \label{eqn:Kprime}
        \begin{aligned}
        K_w'(n,j) &= (w+2)(6j^2 - 6(w+1)(n-j)j + w(w+1)(n-j)^2) \sigma_1(n-j) \\
        &\quad + 15(w+6)(w(n-j) - 4j) \sigma_3(n-j).
        \end{aligned}
    \end{equation}
    For $w=0$, the form $\wtilde{G}_0=\frac{3}{2^{11}\cdot7}$ is a positive constant, so $\tilde{a}_{0,-}^{(0)}>0$ and $\tilde{a}_{1,-}^{(0)}=0$.
    This establishes both assertions at the base weight $w=0$.
    Assume both assertions hold at weight $w$.
    For $0\leq n\leq \frac{w}{4}+1$, the induction hypotheses give $\tilde{a}_{n,-}^{(w)}\geq0$, and \eqref{eqn:kappa_G} gives $\kappa_{2,w}^{\alpha_w}(n)<0$.
    Furthermore, if $0\leq j\leq n-1$, then
    \[
        w(n-j)-2j=(w+2)(n-j)-2n\geq w+2-2n\geq \frac{w}{2}\geq0,
    \]
    so \eqref{eqn:K_G} gives $K_{2,w}^{\alpha_w}(n,j)<0$.
    It follows from \eqref{eqn:Gtilde_coeff_rec} that $\tilde{a}_{n,-}^{(w+4)}\geq0$ for $0\leq n\leq \frac{w}{4}+1$, proving Proposition \ref{prop:Gtilde_pos} at weight $w+4$.

    It remains to prove the boundary assertion at weight $w+4$.
    For $n = \frac{w}{4} + 2$, \eqref{eqn:Gtilde_coeff_linear} becomes
    \begin{equation}
        \tilde{a}_{\frac{w}{4}+2,-}^{(w)} = - \frac{32}{(w+8)(w+12)} \sum_{j=0}^{\frac{w}{4}+1} K_w'\left(\frac{w}{4}+2,j\right) \tilde{a}_{j,-}^{(w)}.
    \end{equation}
    By combining this with $\kappa_{2,w}^{\alpha_w}\left(\frac{w}{4}+2\right) = \frac{w+10}{6}$, \eqref{eqn:Gtilde_coeff_rec} gives
    \begin{align*}
        \tilde{a}_{\frac{w}{4}+2,-}^{(w+4)} &= \frac{3(w+10)(w+14)}{16(w+4)(w+9)(w+11)(w+16)} \left(-\frac{w+10}{6} \tilde{a}_{\frac{w}{4}+2,-}^{(w)} - \sum_{j=0}^{\frac{w}{4}+1} K_{2,w}^{\alpha_w}\left(\frac{w}{4}+2,j\right) \tilde{a}_{j,-}^{(w)}\right) \\
        &= \frac{3(w+10)(w+14)}{16(w+4)(w+9)(w+11)(w+16)} \sum_{j=0}^{\frac{w}{4}+1} \left(\frac{16(w+10)}{3(w+8)(w+12)} K_w' \left(\frac{w}{4}+2,j\right) - K_{2,w}^{\alpha_w}\left(\frac{w}{4}+2,j\right)\right) \tilde{a}_{j,-}^{(w)},
    \end{align*}
    so it is enough to show that
    \begin{equation}
        \label{eqn:Kprime2}
        K_w''(j) := \frac{16(w+10)}{3(w+8)(w+12)} K_w' \left(\frac{w}{4}+2,j\right) - K_{2,w}^{\alpha_w}\left(\frac{w}{4}+2,j\right)
    \end{equation}
    is positive for all $0 \le j \le \frac{w}{4} + 1$.
    Let $r = \frac{w}{4} + 2 - j$, so that $1 \le r \le \frac{w}{4} + 2$.
    Then $K_w''(j)$ can be written in the form
    \begin{equation}
        K_w''(j) = A_w(r) \sigma_1(r) + B_w(r) \sigma_3(r)
    \end{equation}
    where
    \begin{align}
        A_w(r) &= \frac{16(w+10)}{3(w+8)(w+12)}\cdot (w+2)\left(6\left(\frac{w}{4}+2-r\right)^2 - 6(w + 1) r\left(\frac{w}{4}+2-r\right) + w(w+1) r^2\right) \nonumber \\
        &\quad+ 2(w+1)\left(wr- 2\left(\frac{w}{4}+2-r\right)\right) \nonumber \\
        &= \frac{w+4}{3(w+8)(w+12)} (16w^3r^2 - 18w^3 r + 3w^3 + 240 w^2 r^2 - 342 w^2r + 69w^2 \nonumber \\
        &\qquad\qquad\qquad\qquad\qquad + 896wr^2 - 1908 wr + 528 w + 960r^2 - 2592r + 1344) \label{eqn:Aw} \\
        B_w(r) &= \frac{16(w+10)}{3(w+8)(w+12)} \cdot 15(w+6)\left(wr-4\left(\frac{w}{4}+2-r\right)\right) + 5(w+6)(w+16) \nonumber \\
        &= \frac{5(w+4)(w+6)}{(w+8)(w+12)} (w^2 + 16wr + 16w + 160r + 64) \label{eqn:Bw}
    \end{align}
    It is clear that $B_w(r) > 0$ for all $r \ge 1$ from \eqref{eqn:Bw}.
    If $r \ge 2$, substituting $r$ with $u + 2$ in the second factor of \eqref{eqn:Aw} gives
    \[
    16w^3 u^2 + 46w^3 u + 31w^3 + 240 w^2 u^2 + 618w^2u + 345w^2 + 896wu^2 + 1676 wu + 296 w + 960 u^2 + 1248u
    \]
    which is nonnegative.
    When $r = 1$, $\sigma_1(1) = \sigma_3(1) = 1$ and
    \[
    K_w''\left(\frac{w}{4}+1\right) = A_w(1) + B_w(1) = \frac{(w+4)(16w^2 + 409w + 2484)}{3(w+12)} > 0.
    \]
    Hence $K_w''(j) > 0$ for all $0 \le j \le \frac{w}{4} + 1$.
    The induction step for the boundary assertion follows, completing the simultaneous induction.
\end{proof}

\begin{corollary}
    \label{cor:Apos_4mod8}
    For $d \equiv 4 \pmod{8}$, we have
    \[
    \rA_+(d) \le \sqrt{2n_{d,-}} = \sqrt{2\left(\left\lfloor \frac{d}{16} \right\rfloor + 1\right)},
    \]
    where the inequality is strict for $d \ne 12$.
\end{corollary}
\begin{proof}
    As in the proof of Corollary \ref{cor:Apos_0mod8}, Corollary \ref{cor:negeigpos}, Propositions \ref{prop:Gtilde_pos} and \ref{prop:Gtilde_pos_boundary}, and \eqref{eqn:bzero_neg} show that
    \[
        f(\bx) := M_{d,-}(\bx) - M_{d,-}(\mathbf{0}) e^{-\pi \|\bx\|^2}
    \]
    lies in $\mathcal{A}_{+}(d)$ and satisfies $r(f) = r(\what{f}) \le \sqrt{2n_{d,-}}$.

    It remains to prove strictness when $d \neq 12$: if $b_{0,-} > 0$, then $M_{d,-}(\mathbf{0}) = -\pi b_{0,-} < 0$ by \eqref{eqn:Mdnegzero}, so $f(\bx) = -M_{d,-}(\mathbf{0}) e^{-\pi \|\bx\|^2} > 0$ on the sphere $\|\bx\| = \sqrt{2n_{d,-}}$ where $M_{d,-}$ vanishes, and continuity gives $r(f) < \sqrt{2n_{d,-}}$.
    Every term in \eqref{eqn:bzero_neg} is nonnegative, so it suffices to show that the $j = 0$ term $p_{24(n_{d,-} - 1)}(n_{d,-} - 1)\, \tilde{a}_{0,-}^{(w - 12)}$ is positive.
    The partition factor is always positive, and taking $n = 0$ in \eqref{eqn:Gtilde_coeff_rec} gives
    \[
        \tilde{a}_{0,-}^{(w + 4)} = \frac{(w+6)(w+10)(w+14)}{256(w+4)(w+9)(w+11)}\, \tilde{a}_{0,-}^{(w)},
    \]
    so $\tilde{a}_{0,-}^{(0)} = \frac{3}{2^{11} \cdot 7} > 0$ implies that $\tilde{a}_{0,-}^{(w)} > 0$ for all $w \equiv 0 \pmod{4}$ with $w \geq 0$.
    Since \eqref{eqn:wdneg} gives $w_{d,-} \geq 12$ for every $d \equiv 4 \pmod{8}$ with $d \neq 12$, we conclude that $b_{0,-} > 0$, which completes the proof.
    (For $d = 12$, we have $w_{d,-} = 8$ and $\wtilde{G}_{-4} = 0$, so $b_{0,-} = 0$ and $M_{12,-}(\mathbf{0}) = 0$.)
\end{proof}

Combining Corollaries \ref{cor:Apos_0mod8} and \ref{cor:Apos_4mod8} proves Theorem \ref{thm:main}.

\begin{remark}
    Based on experimental evidence, we conjecture that $\widetilde{G}_w$ is completely positive for all $w \equiv 0 \pmod{4}$.
\end{remark}

\section*{Appendix}
\renewcommand{\thesubsection}{\Alph{subsection}}
\renewcommand{\theHsubsection}{appendix.\arabic{subsection}}
\setcounter{subsection}{0}

All Lean and Sage code can be found in the GitHub repository \url{https://github.com/seewoo5/posqmf}.

\subsection{Lean}
\label{subsec:lean}

Some of the results in this paper have been formalized in Lean 4 with the help of \texttt{Claude Opus 5 / Fable 5}.
The main purpose of the formalization is to verify the routine but lengthy computations, so that readers can focus on the main ideas of the proofs.
In particular, we formalized results on Kaneko--Zagier operators (Lemmas \ref{lem:KZ2_coeff}, \ref{lem:KZ3_coeff}, and \ref{lem:intertwine}) and on the positivity of the coefficients of $\wtilde{F}_{w-2}$ and $\wtilde{G}_w$ (Propositions \ref{prop:Ftilde_pos}, \ref{prop:Gtilde_pos}, and \ref{prop:Gtilde_pos_boundary}).
The Lean code can be found under the \texttt{posqmf/lean/QuasiModularForms} and \texttt{posqmf/lean/UncertaintyPrinciple} directories of the GitHub repository.

\subsubsection{Quasimodular forms}

To formalize the families $F_w$, $\wtilde{F}_{w-2}$, and $\wtilde{G}_w$, we first developed the required theory of level 1 quasimodular forms.
\texttt{mathlib} already contains formalizations of modular forms, the weight two Eisenstein series $E_2$, and the derivative and Serre derivative operators $D$ and $\partial_k$ on functions $F : \bH \to \bC$; see \href{https://leanprover-community.github.io/mathlib4_docs/Mathlib/NumberTheory/ModularForms/Derivative.html#Derivative.normalizedDerivOfComplex}{\lstinline|normalizedDerivOfComplex|} and \href{https://leanprover-community.github.io/mathlib4_docs/Mathlib/NumberTheory/ModularForms/Derivative.html#Derivative.serreDerivative}{\lstinline|serreDerivative|}.
However, we chose to introduce different models for the following reasons:
\begin{itemize}
  \item Defining the family $\wtilde{F}_{w-2}$ requires the operator $\delta = \frac{\partial}{\partial E_2}$ on the ring of quasimodular forms, which is difficult to express in the function-theoretic framework above.
  \item We can define $E_2$, $E_4$, and $E_6$ directly as power series ($q$-series), and define $D$ and $\partial_k$ at the level of power series; this is better suited to our purposes.
\end{itemize}
We therefore define $E_2$, $E_4$, and $E_6$ in two ways: as power series and as generators of a polynomial ring.
For the power-series model, we use \eqref{eqn:e2fourier}, \eqref{eqn:e4fourier}, and \eqref{eqn:e6fourier} as the definitions of $E_2$, $E_4$, and $E_6$.
We define the derivative and Serre derivative at the level of power series by $D = q \frac{\dd}{\dd q}$, as in \eqref{eqn:D_def}, and $\partial_k = D - \frac{k}{12} E_2$; these are \lstinline|QExpansion.D| and \lstinline|QExpansion.serreD|, respectively.

\begin{lstlisting}[language=lean]
open ArithmeticFunction Finset PowerSeries
open scoped sigma

namespace QExpansion

noncomputable section

def qSigma (k : ℕ) : ℝ⟦X⟧ := mk fun n ↦ (σ k n : ℝ)

def E₂ : ℝ⟦X⟧ := 1 - (24 : ℝ) • qSigma 1

def E₄ : ℝ⟦X⟧ := 1 + (240 : ℝ) • qSigma 3

def E₆ : ℝ⟦X⟧ := 1 - (504 : ℝ) • qSigma 5

def D (f : ℝ⟦X⟧) : ℝ⟦X⟧ := mk fun n ↦ (n : ℝ) * coeff n f

def serreD (k : ℝ) (f : ℝ⟦X⟧) : ℝ⟦X⟧ := D f - (k / 12) • (E₂ * f)
\end{lstlisting}

For this model, we take Ramanujan's identities \eqref{eqn:ramanujan} as \textbf{axioms}, formalized as \lstinline|ramanujan_E₂|, \lstinline|ramanujan_E₄|, and \lstinline|ramanujan_E₆|.
These are the only additional axioms used in the formalization, beyond Lean's standard axioms \lstinline|propext|, \lstinline|Classical.choice|, and \lstinline|Quot.sound|.

\begin{lstlisting}[language=lean]
namespace QExpansion

axiom ramanujan_E₂ : D E₂ = (1 / 12 : ℝ) • (E₂ * E₂ - E₄)

axiom ramanujan_E₄ : D E₄ = (1 / 3 : ℝ) • (E₂ * E₄ - E₆)

axiom ramanujan_E₆ : D E₆ = (1 / 2 : ℝ) • (E₂ * E₆ - E₄ * E₄)
\end{lstlisting}

For the second model, we use the algebraic independence of $E_2$, $E_4$, and $E_6$ over $\mathbb{C}$ \cite[Lemma~117, p.~70]{martin2005formes}; in particular, the same holds over $\mathbb{R}$. We thus represent the ring they generate by the polynomial ring $\mathbb{R}[E_2,E_4,E_6]$, formalized below as \lstinline|QM|.
Here $D$ is defined from Ramanujan's identities by extending its values on the generators to the entire polynomial ring using \lstinline|mkDerivation|, while $\partial_k$ is defined as $D - \frac{k}{12} E_2$ (\lstinline|PolynomialModel.D| and \lstinline|PolynomialModel.serreD|).
We define $\delta$ as the partial derivative with respect to $E_2$ (\lstinline|PolynomialModel.delta|) and the weight-multiplication operator as the weighted Euler operator $H = 2E_2 \frac{\partial}{\partial E_2} + 4E_4 \frac{\partial}{\partial E_4} + 6E_6 \frac{\partial}{\partial E_6}$ (\lstinline|PolynomialModel.eulerOp|).
We also formalized the relevant $\mathfrak{sl}_2$-relations among $D$, $\delta$, and $H$ as \lstinline|PolynomialModel.sl2_lie_h_e|, \lstinline|PolynomialModel.sl2_lie_h_f|, and \lstinline|PolynomialModel.sl2_lie_e_f|, as well as Lemma \ref{lem:delta_serre} as \lstinline|PolynomialModel.delta_serreD|.

\begin{lstlisting}[language=lean]
open MvPolynomial

namespace PolynomialModel

abbrev QM := MvPolynomial (Fin 3) ℝ

noncomputable section

def E₂ : QM := X 0

def E₄ : QM := X 1

def E₆ : QM := X 2

def dGen : Fin 3 → QM
  | 0 => (1 / 12 : ℝ) • (E₂ * E₂ - E₄)
  | 1 => (1 / 3 : ℝ) • (E₂ * E₄ - E₆)
  | 2 => (1 / 2 : ℝ) • (E₂ * E₆ - E₄ * E₄)

def D : Derivation ℝ QM QM := mkDerivation ℝ dGen

def serreD (k : ℝ) (G : QM) : QM := D G - (k / 12 : ℝ) • (E₂ * G)

def delta : Derivation ℝ QM QM := pderiv 0

def eulerOp : Derivation ℝ QM QM :=
  ((2 : ℝ) • E₂) • pderiv 0 + ((4 : ℝ) • E₄) • pderiv 1 + ((6 : ℝ) • E₆) • pderiv 2

def HasWeight (G : QM) (w : ℝ) : Prop := eulerOp G = w • G

theorem sl2_lie_h_e : ⁅eulerOp, D⁆ = (2 : ℝ) • D := by ...

theorem sl2_lie_h_f : ⁅eulerOp, (-12 : ℝ) • delta⁆ = (-2 : ℝ) • ((-12 : ℝ) • delta) := by ...

theorem sl2_lie_e_f : ⁅D, (-12 : ℝ) • delta⁆ = eulerOp := by ...

theorem delta_serreD {G : QM} {w : ℝ} (k : ℝ) (h : HasWeight G w) :
  delta (serreD k G) = serreD k (delta G) + ((w - k) / 12 : ℝ) • G := by ...
\end{lstlisting}

For an element of \lstinline|QM|, one obtains its $q$-expansion by evaluating it at the $q$-expansions of $E_2$, $E_4$, and $E_6$.
This map connects the two models and is formalized as \lstinline|PolynomialModel.qexp|.
The theorems \lstinline|PolynomialModel.qexp_D| and \lstinline|PolynomialModel.qexp_serreD| show that the two definitions of $D$ and $\partial_k$ are compatible under the $q$-expansion map.

\begin{lstlisting}[language=lean]
namespace PolynomialModel

def qexp : QM →ₐ[ℝ] PowerSeries ℝ :=
  aeval ![QExpansion.E₂, QExpansion.E₄, QExpansion.E₆]

theorem qexp_dGen (i : Fin 3) : qexp (dGen i) = QExpansion.D (qexp (X i)) := by ...

@[simp] theorem qexp_D (p : QM) : qexp (D p) = QExpansion.D (qexp p) := by ...

@[simp] theorem qexp_serreD (k : ℝ) (p : QM) :
    qexp (serreD k p) = QExpansion.serreD k (qexp p) := by ...
\end{lstlisting}

\subsubsection{Kaneko--Zagier operators}

We formalized Lemmas \ref{lem:KZ2_coeff}, \ref{lem:KZ3_coeff}, and \ref{lem:intertwine}.
The second- and third-order Kaneko--Zagier operators $L_{2,k}^{\alpha}$ and $L_{3,k}^{(\alpha,\beta)}$ on power series, defined in \eqref{eqn:KZ2_def} and \eqref{eqn:KZ3_def}, are formalized as \lstinline|KanekoZagier.L₂| and \lstinline|KanekoZagier.L₃|, respectively.

\begin{lstlisting}[language=lean]
open PowerSeries QExpansion

namespace KanekoZagier

noncomputable section

def L₂ (k α : ℝ) (f : ℝ⟦X⟧) : ℝ⟦X⟧ :=
  D (D f) - ((k + 1) / 6) • (E₂ * D f) + (k * (k + 1) / 12) • (D E₂ * f) + α • (E₄ * f)

def L₃ (k α β : ℝ) (f : ℝ⟦X⟧) : ℝ⟦X⟧ :=
  D (D (D f)) - ((k + 2) / 4) • (E₂ * D (D f))
    + (((k + 1) * (k + 2) / 4) • D E₂ + α • E₄) * D f
    - ((k * (k + 1) * (k + 2) / 24) • D (D E₂) + (k * α / 4) • D E₄ - β • E₆) * f
\end{lstlisting}

The theorems \lstinline|KanekoZagier.coeff_L₂| and \lstinline|KanekoZagier.coeff_L₃| formalize Lemmas \ref{lem:KZ2_coeff} and \ref{lem:KZ3_coeff}, respectively; they compute the coefficients of $L_{2,k}^{\alpha}G$ and $L_{3,k}^{(\alpha,\beta)}G$ in terms of the coefficients of $G$.

\begin{lstlisting}[language=lean]
open ArithmeticFunction Finset PowerSeries QExpansion
open scoped sigma

namespace KanekoZagier

noncomputable section

def κ₂ (k α : ℝ) (n : ℕ) : ℝ := (n : ℝ) ^ 2 - (k + 1) / 6 * n + α

def K₂ (k α : ℝ) (n j : ℕ) : ℝ :=
  2 * (k + 1) * (2 * (j : ℝ) - k * ((n : ℝ) - (j : ℝ))) * (σ 1 (n - j) : ℝ)
    + 240 * α * (σ 3 (n - j) : ℝ)

theorem coeff_L₂ (k α : ℝ) (G : ℝ⟦X⟧) (n : ℕ) :
    coeff n (L₂ k α G) = κ₂ k α n * coeff n G + ∑ j ∈ range n, K₂ k α n j * coeff j G := by ...

def κ₃ (k α β : ℝ) (n : ℕ) : ℝ := (n : ℝ) ^ 3 - (k + 2) / 4 * (n : ℝ) ^ 2 + α * n + β

def K₃ (k α β : ℝ) (n j : ℕ) : ℝ :=
  (k + 2) * (6 * (j : ℝ) ^ 2 - 6 * (k + 1) * ((n : ℝ) - (j : ℝ)) * (j : ℝ)
      + k * (k + 1) * ((n : ℝ) - (j : ℝ)) ^ 2) * (σ 1 (n - j) : ℝ)
    + 60 * α * (4 * (j : ℝ) - k * ((n : ℝ) - (j : ℝ))) * (σ 3 (n - j) : ℝ)
    - 504 * β * (σ 5 (n - j) : ℝ)

theorem coeff_L₃ (k α β : ℝ) (G : ℝ⟦X⟧) (n : ℕ) :
    coeff n (L₃ k α β G)
      = κ₃ k α β n * coeff n G + ∑ j ∈ range n, K₃ k α β n j * coeff j G := by ...
\end{lstlisting}

Finally, the theorem \lstinline|KanekoZagier.L₃_comp_L₂_eq_L₂_comp_L₃| formalizes the intertwining relation \eqref{eqn:intertwine} under the parameter constraints \eqref{eqn:intertwine_constraints}.

\begin{lstlisting}[language=lean]
namespace KanekoZagier

open PowerSeries QExpansion

noncomputable section

def shiftA (k α : ℝ) : ℝ := α - (3 * k ^ 2 + 36 * k + 104) / 144

def shiftB (k α β : ℝ) : ℝ := β + (k + 4) / 12 * α - (k + 4) ^ 2 * (k + 7) / 864

def shiftC (k γ : ℝ) : ℝ := γ - k * (k + 2) / 144

def shiftA' (k α' : ℝ) : ℝ := α' - (3 * k ^ 2 + 12 * k + 8) / 144

def shiftB' (k α' β' : ℝ) : ℝ := β' + k / 12 * α' - k ^ 2 * (k + 3) / 864

def shiftC' (k γ' : ℝ) : ℝ := γ' - (k + 6) * (k + 8) / 144

theorem L₃_comp_L₂_eq_L₂_comp_L₃ (k α β γ α' β' γ' : ℝ)
    (h₁ : shiftA k α + shiftC k γ = shiftA' k α' + shiftC' k γ')
    (h₂ : shiftB k α β - shiftC k γ = shiftB' k α' β' - 2 * shiftA' k α' / 3)
    (h₃ : shiftC k γ * (shiftA k α + 1 / 2)
      = shiftA' k α' * (shiftC' k γ' + 1 / 6) - shiftB' k α' β')
    (h₄ : shiftC k γ * (shiftB k α β - shiftA k α / 3 - 1 / 9)
      = shiftB' k α' β' * (shiftC' k γ' + 1 / 3))
    (f : ℝ⟦X⟧) :
    L₃ (k + 4) α β (L₂ k γ f) = L₂ (k + 6) γ' (L₃ k α' β' f) := by ...
\end{lstlisting}

The proofs of Lemmas \ref{lem:KZ2_coeff}, \ref{lem:KZ3_coeff}, and \ref{lem:intertwine} are purely algebraic; the formal proofs are handled largely by the tactics \lstinline|simp|, \lstinline|ring|, \lstinline|ring_nf|, and \lstinline|module|.

\subsubsection{Log-polynomial inequalities}

We also formalized the log-polynomial inequalities \eqref{eqn:logineq1}, \eqref{eqn:logineqY4}, \eqref{eqn:logineqY8}, and \eqref{eqn:logineqY10}, following the arguments used in the proofs of Lemma \ref{lem:logineq1} and Theorem \ref{thm:Ywpos}.

\begin{lstlisting}[language=lean]
theorem logIneq_basic {x : ℝ} (hx : 0 < x) : 2 * x / (x + 2) < log (1 + x) := by ...

theorem logIneq_Y₄ {x : ℝ} (hx : 0 < x) :
    x * (x + 2) / (2 * (x ^ 2 + x + 1)) < log (1 + x) := by ...

theorem logIneq_Y₈ {x : ℝ} (hx : 0 < x) :
    (11 * x ^ 4 + 28 * x ^ 3 + 18 * x ^ 2 + 12 * x) / 12 < (x ^ 2 + x + 1) ^ 2 * log (1 + x) := by ...

theorem logIneq_Y₁₀ {x : ℝ} (hx : 1 ≤ x) :
    (2 * x ^ 5 + 11 * x ^ 4 - x ^ 3 - 24 * x ^ 2 - 12 * x) / 6
      < -((x ^ 2 + x + 1) * (2 * x + 1) * (x + 2) * (1 - x)) * log (1 + x) := by ...
\end{lstlisting}

\subsubsection{Positivity of the coefficients of $\wtilde{F}_{w-2}$ and $\wtilde{G}_{w}$}

We first defined $F_w$ and $\wtilde{F}_{w-2}$ for $w \equiv 0 \pmod{4}$ as elements of \lstinline|QM| by defining $F_w$ recursively from $F_8$ using \eqref{eqn:poseig4}, and then setting $\wtilde{F}_{w-2}=\delta F_w$.
Within the namespace \lstinline|UncertaintyPrinciple|, \lstinline|fFam N| corresponds to $F_{4N}$, and \lstinline|ftildeFam N| corresponds to $\wtilde{F}_{4N-2}$.
The theorems \lstinline|coeff_ftildeSeries_pos| and \lstinline|coeff_ftildeSeries_boundary| formalize Proposition \ref{prop:Ftilde_pos}.
For convenience, we set $F_0$ and $F_4$ to zero; neither is used in the proof of Proposition \ref{prop:Ftilde_pos}.

\begin{lstlisting}[language=lean]
open Finset PowerSeries

namespace UncertaintyPrinciple

noncomputable section

section QExpansion

open ArithmeticFunction QExpansion PolynomialModel KanekoZagier
open scoped sigma

def cF (w : ℝ) : ℝ := 3 * (w - 4) * w / (16 * (w - 10) * (w - 5) * (w - 3) * (w + 2))

def SFp (w : ℝ) (G : QM) : QM :=
  ((w - 6) * (w - 5) / 36 : ℝ) • (E₄ * G) - serreD w (serreD (w - 2) G)

def F₈ : QM := (1 / 1728 : ℝ) • (E₂ ^ 2 * E₄ - (2 : ℝ) • (E₂ * E₆) + E₄ ^ 2)

def fFam : ℕ → QM
  | 0 => 0
  | 1 => 0
  | 2 => F₈
  | N + 3 => cF (4 * N + 8) • SFp (4 * N + 8) (fFam (N + 2))

def ftildeFam (N : ℕ) : QM := delta (fFam N)

def ftildeSeries (N : ℕ) : ℝ⟦X⟧ := qexp (ftildeFam N)

theorem coeff_ftildeSeries_pos (N j : ℕ) (hN : 2 ≤ N) (hj1 : 1 ≤ j) (hj : j ≤ N - 2) :
    0 < coeff j (ftildeSeries N) := by ...

theorem coeff_ftildeSeries_boundary (N : ℕ) (hN : 3 ≤ N) :
    1 / 360 ≤ coeff (N - 2) (ftildeSeries N) := by ...
\end{lstlisting}

For the $(-1)^{d/4+1}$-family, we formalized only $\wtilde{G}_w$ for $w \equiv 0 \pmod{4}$ as elements of \lstinline|QM|, defining the family recursively from $\wtilde{G}_0$ using \eqref{eqn:Gtilde_rec}.
Formalizing $G_w$ would require more work because it involves Jacobi theta functions and the logarithm of the modular lambda function.
The theorem \lstinline|coeff_gtildeSeries_nonneg| formalizes Propositions \ref{prop:Gtilde_pos} and \ref{prop:Gtilde_pos_boundary}, while \lstinline|coeff_gtildeSeries_zero_pos| establishes the additional strict positivity of the constant coefficient.

\begin{lstlisting}[language=lean]
open Finset PowerSeries

namespace UncertaintyPrinciple

noncomputable section

section QExpansion

open ArithmeticFunction QExpansion KanekoZagier
open scoped sigma

def cG (w : ℝ) : ℝ := 3 * (w + 10) * (w + 14) / (16 * (w + 4) * (w + 9) * (w + 11) * (w + 16))

def SGp (w : ℝ) (G : QM) : QM :=
  ((w + 8) * (w + 9) / 36 : ℝ) • (E₄ * G) - serreD (w + 2) (serreD w G)

def gtildeFam : ℕ → QM
  | 0 => (3 / 14336 : ℝ) • 1
  | N + 1 => cG (4 * N) • SGp (4 * N) (gtildeFam N)

def gtildeSeries (N : ℕ) : ℝ⟦X⟧ := qexp (gtildeFam N)

theorem coeff_gtildeSeries_nonneg (N : ℕ) : ∀ n, n ≤ N + 1 → 0 ≤ coeff n (gtildeSeries N) := by ...

theorem coeff_gtildeSeries_zero_pos (N : ℕ) : 0 < coeff 0 (gtildeSeries N) := by ...
\end{lstlisting}

\subsection{Sage}
\label{subsec:sage}

We also implemented in Sage several computations involving quasimodular forms, building on code developed by the author for \cite[Appendix~A]{lee2024algebraic}.
In particular, we implemented the families $F_w$, $\wtilde{F}_{w-2}$, $G_w$, $Y_w$, and $\wtilde{G}_w$, together with the Kaneko--Zagier operators $L_{2,k}^{\alpha}$ and $L_{3,k}^{(\alpha,\beta)}$.
We then verified, up to weight $w \le 100$, that these families satisfy the corresponding recurrence relations and MLDEs.
Some of the results are formalized in Lean as described in Appendix \ref{subsec:lean}, where you can consider the Sage code as an independent sanity check of the results.

To implement $G_w$ and $Y_w$, we defined \texttt{QM2} as a polynomial ring with three generators corresponding to $H_2$, $H_4$, and $E_2$, and then set \texttt{QM2\_LS = QM2['LS']}, where the new generator \texttt{LS} represents $\cL_S$.
In this extended ring, we implemented $q$-expansions and ordinary and Serre derivatives using \eqref{eqn:LS_qexp} and \eqref{eqn:LS_derivative}.
The implementation is contained in the Jupyter notebook \texttt{uncertainty\_principle.ipynb} and the Sage files under the directory \texttt{posqmf/sage} in the same GitHub repository.

\bibliographystyle{plain}
\bibliography{refs}
\end{document}